\documentclass[reqno]{amsart}
\usepackage{xcolor}
\usepackage{amssymb,amsmath,amsthm,mathrsfs, array}
\usepackage{amsfonts,latexsym,amsxtra,amscd}
\usepackage{hyperref}
\usepackage{enumerate}
\usepackage{graphicx, stmaryrd}
\usepackage[centering]{geometry}
\AtEndDocument{

\bigskip{\footnotesize%
\textsc{School of Mathematics and Statistics, Wuhan University of Technology, Wuhan, 430070, P. R. China} \par
\textit{E-mail address} \texttt{: chrischang2016@gmail.com}}

\bigskip{\footnotesize%
\textsc{School of Mathematics, South China University of Technology, Wushan
Road 381, Tianhe District, Guangzhou, China} \par
\textit{E-mail address} \texttt{: scbingli@scut.edu.cn}}

\bigskip{\footnotesize%
\textsc{Department of Mathematical Sciences, P.O. Box 3000, 90014 University of Oulu, Finland} \par
\textit{E-mail address} \texttt{: meng.wu@oulu.fi}}
}

\newtheorem{thm}{Theorem}[section]
\newtheorem{prop}[thm]{Proposition}

\newtheorem{lem}[thm]{Lemma}
\newtheorem{defn}[thm]{Definition}
\newtheorem{rem}[thm]{Remark}

\newtheorem{con}[thm]{Conjecture}

\numberwithin{equation}{section}

\renewcommand\bigskip{\medskip}

\def\to{\rightarrow}

\def\N{\mathbb N}

\DeclareMathOperator{\diam}{diam}

\begin{document}

%%%%%%%%%%%%%%%%%%%%%%%%%%%%%%%%%%%%%%%%%%%%%%%%%%%%%%
\title[Orbit complexity intermediate value Furstenberg problem]
{On orbit complexity of dynamical systems: intermediate value property and level set related to a Furstenberg problem}
%%%%%%%%%%%%%%%%%%%%%%%%%%%%%%%%%%%%%%%%%%%%%%%%%%%%%%


\author{}
%\thanks{xxx}

\author{}
\address{}
\email{}
%\thanks{xxx}

\author{Yuanyang Chang,  Bing Li, and Meng Wu}
\date{}

\begin{abstract}
  For symbolic dynamics with some mild conditions, we solve the lowering topological entropy problem for subsystems and determine the Hausdorff dimension of the level set with given complexity, where the complexity is represented by Hausdorff dimension of orbit closure. These results can be applied to some dynamical systems such as $\beta$-transformations, conformal expanding repeller, etc. We also determine the dimension of the Furstenberg level set, which is related to a problem of Furstenberg on the orbits under two multiplicatively independent maps.
\end{abstract}

\maketitle

%%%%%%%%%%%%%%%%%%%%%%%%%%%%%%%%%%%%%%%%%%%%%%%%%%%%%%%%%%%%%%%%%%%
%%%%%%%%%%%%%%%%%%%%%%%%%%%%%%%%%%%%%%%%%%%%%%%%%%%%%%%%%%%%%%%%%%%
\section{Introduction}
\subsection{Orbit complexity: symbolic setting}
Let $(X, T)$ be a \emph{dynamical system}, where $X$ is a compact metric space and $T: X\rightarrow X$ is a continuous map. Topological entropy is a quantity to characterize the complexity of a dynamical system, see \cite{Bow73} for the definition of topological entropy for non-compact subsets. In \cite{SW91}, Shub and Weiss posed the \emph{lowering topological entropy problem}: for a dynamical system with topological entropy $h_{top}(T, X)>0$, and for any $h\in [0,h_{top}(T, X)]$, can we always find a factor $(Y, S)$ with entropy $h$? Here a dynamical system $(Y, S)$ is said to be a factor of $(X, T)$, if there exists a continuous surjection $\pi: X\rightarrow Y$ satisfying $S\circ\pi=\pi\circ T$. Lindenstrauss \cite{Lin95} gave an affirmative answer to this problem when the space $X$ is of finite dimension; while for infinite dimensional dynamical systems, he provided a negative answer by constructing a minimal system. Motivated by this and the intermediate value property of Hausdorff dimension of Borel subsets, Huang et al \cite{HYZ10, HYZ14} studied a similar question, called the \emph{lowering topological entropy problem for subsets}, that is, for any $h\in [0,h_{top}(T, X)]$, whether there exists a non-empty compact subset $Y\subset X$ with entropy $h$? Note that here the subset $Y$ is not necesssarily invariant, and they showed this is the case for any dynamical systems with finite topological entropy.

In this paper, we are interested in a related but different problem. Rather than finding factors or non-empty compact subsets of intermediate entropy, we are concerned with the problem whether there exists a subsystem with any prescribed intermediate entropy, namely, \emph{lowering topological entropy problem for subsystems}. Here, by a \emph{subsystem} we mean a pair $(Y, T)$, where $Y$ is a non-empty closed invariant subset; and $Y$ is called \emph{minimal} if it does not contain any non-empty closed invariant subsets. The subsystem $(Y, T)$ is called \emph{minimal} if $Y$ is minimal. Since a subsystem $(Y, T)$ is a factor itself, thus finding a subsystem with prescribed intermediate entropy is more restrictive than finding a factor or a non-empty compact subset. To rule out some trivial cases, we assume that $(X, T)$ is neither  minimal nor uniquely ergodic, where in both cases no proper subsystem exists.
Our first main result confirms the existence of subsystems with prescribed intermediate entropy for symbolic dynamics which satisfy a weak form of specification property and certain entropy gap condition. The symbolic dynamics we consider include the $\beta$-shifts, S-gap shifts and more general coded systems. See Section 2 for more details about the notations and definitions. Throughout, we use $\mathcal{O}_{T}(x)$ to denote the orbit of $x\in X$ under $T$ and $\overline{\mathcal{O}_{T}(x)}$ to denote the orbit closure. For a collection
$\mathcal{D}$ of words, use $h_{top}(\mathcal{D})$ to denote the exponential rate of growth of the number of words of length $n$ in $\mathcal{D}$ (see Section 2.2 for more details).

\begin{thm}\label{int-entropy}
  Let $(\Sigma, \sigma)$ be a subshift over a finite alphabet with positive topological entropy. Suppose the language of $\Sigma$ admits the decomposition
  $\mathcal{L}=\mathcal{C}^p\mathcal{G} \mathcal{C}^s$ and the following conditions hold:
\begin{enumerate}
  \item $\mathcal{G}$ has the specification property;
  \item $\Sigma$ satisfies the entropy gap condition, i.e., $h_{top}(\mathcal{C}^p\cup\mathcal{C}^s)<h_{top}(\sigma, \Sigma)$.
\end{enumerate}
Then for any $h\in [0, h_{top}(\sigma, \Sigma)]$, there exists $w\in\Sigma$ with
$h_{top}(\sigma, \overline{\mathcal{O}_{\sigma}(w)})=h$. %Moreover, if $h\in (0, h_{top}(\sigma, \Sigma))$, then $w$ can be chosen such that $\overline{\mathcal{O}_{\sigma}(w)}$ is minimal.
\end{thm}

%Indeed, from the proof of Theorem \ref{int-entropy}, we have
%\begin{rem}\label{diversity}
%Moreover, $\Sigma'$ can be made minimal and we have
%\begin{equation}\label{orbit_entropy}
%   h_{top}\left(\left\{x\in\Sigma: h_{top}(\overline{\mathcal{O}_{\sigma}}(x))\leq\alpha%\right\}\right)=\alpha.
%\end{equation}
%Here $\overline{\mathcal{O}_{\sigma}}(x)$ denotes the orbit closure of $x$ under the shift map $\sigma$.
%\end{rem}

The proof of Theorem \ref{int-entropy} is constructive, indeed, we can construct a bunch of points with the intermediate entropy property. Since $h_{top}(\sigma, \overline{\mathcal{O}_{\sigma}(w)})$ describes the complexity of the subsystem $(\overline{\mathcal{O}_{\sigma}(w)}, \sigma)$, Theorem \ref{int-entropy} indicates the existence of orbits with any prescribed intermediate complexity.
Then the following question arises naturally: how large is the set of points in $\Sigma$ with the prescribed orbit complexity? To give a better understanding of this question, we need to use Hausdorff dimension, a prevalent notion in fractal geometry, see Section \ref{Hausdorff} for more details on the definition. By the well known Ledrappier-Young formula, if the Hausdorff dimension is defined with respect to the metric
\begin{equation}\label{usualmetric}
d(w, w')=m^{-\min\{i\geq 0: w_{i+1}\neq w'_{i+1}\}}, \forall w, w'\in\Sigma,
\end{equation}
then $\dim_{\rm H} Y=\frac{h_{top}(T, Y)}{\log m}$ holds for any closed invariant subset $Y$, see also \cite[Proposition \uppercase\expandafter{\romannumeral3}. 1]{Fur67}. The above question on diversity of orbit complexity can be rephrased as: \emph{for any $\alpha\in [0, \dim_{\rm H} \Sigma]$, what is the Hausdorff dimension of the level set}
\begin{equation*}
  E_{\sigma}(\alpha):=\left\{w\in \Sigma: \dim_{\rm H} \overline{\mathcal{O}_{\sigma}(w)}=\alpha\right\}?
\end{equation*}

Our second theorem answers the the above question by determining the Hausdorff dimension of
$E_{\sigma}(\alpha)$. We remark that the Hausdorff dimension is defined via a proper metric more general than that in \eqref{usualmetric} (see Section \ref{setting} for more details on notations), which allows us to apply the theorem to some concrete dynamical systems. We shall present these applications in Section 5.
Let $(\Sigma, \sigma)$ be a subshift and $\varphi$ be a strictly positive continuous function on $\Sigma$. We use $\text{Var}_n\varphi$, $d_{\varphi}$, $\dim^\varphi_{\rm H}$ and $Z_n(\sigma, \mathcal{C}^p\cup\mathcal{C}^s, -\gamma\varphi)$ to denote the $n$-variation of $\varphi$, the metric induced by $\varphi$, the
$\varphi$-Hausdorff dimension and the quantity defined in \eqref{Z-n} respectively, and $\gamma=\dim^\varphi_{\rm H} \Sigma$.

\begin{thm}\label{mainthm}
Let $(\Sigma, \sigma)$ be a subshift over a finite alphabet and
$\varphi$ be a strictly positive continuous function on $\Sigma$ satisfying $\sum_{n\geq 1}\text{Var}_n\varphi<\infty$. Equip $\Sigma$ with the metric $d_{\varphi}$ and write $\gamma=\dim^\varphi_{\rm H} \Sigma$. Write
\begin{equation*}
  E_{\sigma}^{\varphi}(\alpha):=\left\{w\in \Sigma: \dim^\varphi_{\rm H} \overline{\mathcal{O}_{\sigma}(w)}=\alpha\right\}
\end{equation*}
Suppose that there exists a decomposition $\mathcal{L}=\mathcal{C}^p\mathcal{G} \mathcal{C}^s$ satisfying:
    \begin{enumerate}
    \item $\mathcal{G}$ has the specification property;
    \item $\Sigma$ satisfies the pressure gap condition, i.e., $\sum_{n\geq 1} Z_n(\sigma, \mathcal{C}^p\cup\mathcal{C}^s, -\gamma\varphi)<\infty$.
\end{enumerate}
Then for any $\alpha\in[0, \gamma]$, we have
\begin{equation}\label{orbitdim}
  \dim^\varphi_{\rm H} E_{\sigma}^{\varphi}(\alpha)=\alpha.
\end{equation}
\end{thm}

%\begin{rem}
%By Theorem \ref{GP} (1), any closed invariant subset of $\Sigma$ has equal Hausdorff and box dimension. The following assertion follows from Theorem \ref{mainthm} immediately:
%\begin{equation}
%  \dim^\varphi_{\rm H} E_{\sigma}(\alpha)=\dim^\varphi_{\rm H} F_{\sigma}(\alpha)=\alpha.
%\end{equation}
%\end{rem}

%Relations between two main theorems: Theorem 1.1 is a special case of Theorem 1.3, can be seen as a warm up; One merit of Theorem 1.1 is that by combinatorial tricks, we can ensure that the subsystem we construct is minimal, which we cannot ensure by ergodic method.
\begin{rem}
   Let $(X, T)$ be a diffeomorphism on a compact Riemannian manifold. Denote the set of $T$-invariant and ergodic measure by $\mathcal{M}_{e}(X, T)$. Katok conjectured that $(X, T)$ has the intermediate entropy property, that is, for every $\alpha\in [0, h_{top}(T, X))$, there exist $\mu\in\mathcal{M}_{e}(X, T)$ such that the measure-theoretic entropy $h(\mu, T)$ equals to $\alpha$. See \cite{Sun21} and references therein for partial results on this conjecture. In the proof of Theorem \ref{mainthm}, we verified this conjecture for a subshift $(\Sigma, \sigma)$ satisfying our assumptions. In fact, we proved even more: for $\mu\in\mathcal{M}_{e}(X, T)$ with $h(\mu, T)=\alpha$, and for almost every point $w\in \text{supp}(\mu)$, the orbit closure of $w$ has the same Hausdorff dimension with $\mu$.
\end{rem}

For the proof of Theorem \ref{mainthm}, we need the following result which has its  own interest.
\begin{thm}\label{intermediate-dimension}
 Let $\Sigma$ and $\varphi$ be defined as in Theorem \ref{mainthm}. Equip $\Sigma$ with the metric $d_{\varphi}$ and write $\gamma=\dim^\varphi_{\rm H} \Sigma$.
 For any $\alpha\in [0, \gamma]$, there exists a closed invariant subset $X$ with a $\sigma$-invariant measure $\nu$ supported on it and satisfying $\dim_{\rm H}^{\varphi}X=\dim_{\rm H}^{\varphi}\nu=\alpha$.
\end{thm}

%We remark that Theorem \ref{intermediate-dimension} is an extension of Theorem \ref{int-entropy} in the sense that if we take $\varphi \equiv 1$ in Theorem \ref{intermediate-dimension}, then we recover Theorem \ref{int-entropy}.

\subsection{Orbit complexity: single dynamical system}
In this subsection, we shall apply Theorem \ref{mainthm} to $\beta$-transformations and conformal expanding repellers. There may be more applications, but we will restrict ourselves to these two well known examples.
\subsubsection{$\beta$-transformations}
Let $\beta>1$ be a real number, the $\beta$-transformation $T_\beta: [0, 1)\to [0, 1)$ is defined by $T_\beta(x)=\beta x\; (\text{mod}\ 1)$ for $x\in [0, 1)$. It is well known that $\beta$-transformation is a typical example of expanding non-finite Markov systems whose properties are reflected by the greedy $\beta$-expansion of 1, and there exists a unique invariant measure called Parry measure which is equivalent to the Lebesgue measure (see the pioneer work \cite{Par60, Ren57} for more details).

The corresponding symbolic space $\Sigma_\beta$, taking closure of the set of $\beta$-expansions of real numbers in $[0, 1)$, is a subshift of $(\mathcal{A}_\beta^\mathbb{N},\sigma)$, where $\mathcal{A}_\beta=\{0,1,\dots,\beta-1\}$ for $\beta\in\mathbb{N}$ and  $\mathcal{A}_\beta=\{0,1,\dots,\lfloor\beta\rfloor\}$ for $\beta\notin\mathbb{N}$. Generally, the subshift $(\Sigma_\beta, \sigma)$ may not satisfy the specification property. Schmeling \cite{Schmeling1997} proved that the set of $\beta>1$ with $\Sigma_\beta$ having specification property is of zero Lebesgue measure and full Hausdorff dimension. We will see that the corresponding symbolic dynamics $(\Sigma_\beta, \sigma)$ satisfies the weak specification property described in Theorem \ref{mainthm} (see Proposition \ref{beta-spec}).

By Birkhoff's ergodic theorem, for Lebesgue almost every $x\in [0, 1)$, the orbit $\mathcal{O}_{T_{\beta}}(x)$ is dense, thus the orbit closure has Hausdorff dimension one. For each $\alpha\in [0, 1]$, write
$$E_{T_\beta}(\alpha)=\big\{x\in [0, 1): \dim_{\rm H} \overline{\mathcal{O}_{T_{\beta}}(x)}=\alpha\big\},$$
here $\dim_{\rm H}$ is defined with respect to the Euclidean metric on $[0, 1]$.
Then we have the following theorem:
\begin{thm}\label{beta}
	Let $\beta>1$ and $\alpha\in [0,1]$. Then $\dim_{\rm H}E_{T_\beta}(\alpha)=\alpha.$
\end{thm}

We remark that the integer case of Theorem \ref{beta} was proved by Feng \cite{Feng}.

\subsubsection{Conformal expanding repellers}
   Let $M$ be a connected $C^1$ Riemannian manifold. Let $U$ be an open subset of $M$ and $T: U\rightarrow M$ be a $C^{1}$ map. If the following holds:
\begin{enumerate}[(1)]
  \item the map $T: U\rightarrow M$ is conformal in the sense that the derivative $T'(x)$ is a scalar multiple of an orthogonal matrix;
  \item there exist a compact set $X\subset U$ and an open neighborhood $V$ of $X$ such that $T(X)\subset X$ and $T^{-1}(X)\cap V\subset X$;
  \item there exist constants $C>0$ and $\lambda>1$ such that
    \begin{equation*}
      \|(T^n)'(x)(u)\|\geq C\lambda^n \|u\|
    \end{equation*}
  for all $x\in X, u\in T_x M$ and $n\geq 1$, here $\|\cdot\|$ denotes the Euclidean norm;
\end{enumerate}
then $(X, T)$ is called a conformal expanding repeller, abbreviated as CER.

For each $\alpha\in [0, \dim_{\rm H}X]$, write
$$E_{T}(\alpha)=\big\{x\in X: \dim_{\rm H} \overline{\mathcal{O}_{T}(x)}=\alpha\big\},$$
here $\dim_{\rm H}$ is defined with respect to the natural metric on $X$.
\begin{thm}\label{CER}
  Let $(X, T)$ be CER whose symbolic extension satisfies the specification property. Then we have
$\dim_{\rm H}E_T(\alpha)=\alpha$ for any $\alpha\in [0, \dim_{\rm H} X]$.
\end{thm}

\subsection{Orbit complexity: two dynamical systems}
Another motivation of this paper is the well known Furstenberg conjecture on the sum of dimensions of orbits under two dynamical systems which exhibit some kind of independence \cite{Fur69}. Here we only consider the most simple $p$-transformation $T_p$ on the unit interval $[0, 1)$, where $p\geq 2$ is an integer. Two integers $p, q\geq 2$ are called \emph{multiplicatively independent} if $\log p/\log q\notin\mathbb{Q}$, or equivalently, there exist no $a, b\in\mathbb{N}$ such that $p^a=q^b$, we denote by $p\nsim q$; otherwise, $p$ and $q$ are called \emph{multiplicatively dependent} and denoted by $p\sim q$.
Furstenberg \cite{Fur67} conjectured that
\begin{equation*}
\dim_{\rm H}\overline{\mathcal{O}_{T_p}(x)}+\dim_{\rm H}\overline{\mathcal{O}_{T_q}(x)} \geq 1
\end{equation*}
for any $x\in [0, 1)\setminus\mathbb{Q}$ under the assumption that $p\nsim q$.
Shmerkin \cite{Shm19} and Wu \cite{Wu19} independently proved that $$\dim_{\rm H}\{x\in [0,1)\setminus\mathbb{Q}: \dim_{\rm H}\overline{\mathcal{O}_{T_p}(x)}+\dim_{\rm H}\overline{\mathcal{O}_{T_q}(x)}< 1\}=0$$
assuming $p\nsim q$. In fact they proved that the above set is a countable union of sets with upper box dimension zero. A natural question is to study the following set
$$F_{p,q}^s:= \{x\in [0, 1): \dim_{\rm H}\overline{\mathcal{O}_{T_p}(x)}+\dim_{\rm H}\overline{\mathcal{O}_{T_q}(x)} =s\},$$
which we call the \emph{Furstenberg level set}, here $0\leq s\leq 2$. Then we can ask
 $$\dim_{\rm H}F_{p,q}^s=?$$
Since the Lebesgume measure $\lambda$ is ergordic with respect to $T_p$ for any integer $p\geq 2$, we know that $\lambda$-a.e. $x\in [0,1)$, $\overline{\mathcal{O}_{T_p}(x)}=[0,1]$, which implies that $\lambda(F_{p,q}^2)=1$.

The following result indicates that the situations are quite different in the multiplicatively dependent and independent cases.
\begin{thm}\label{Furstenbergset}
Let $p, q\geq 2$ be integers and $s\in [0, 2]$. Then
   $$\dim_{\rm H} F_{p, q}^s=\begin{cases}
   \frac{s}{2},\ \  &\ \text{if}\  p\sim q;\\
   \max\{0, s-1\},\ \  &\ \text{if}\ p\nsim q.
   \end{cases}$$
 \end{thm}

Throughout the paper, for functions $f$ and positive-valued functions $g$, we write $f\lesssim g$ if there exists a constant $C$ such that $|f|\leq C g$, and write $f\sim g$ if $f\lesssim g$ and $g\lesssim f$ simultaneously.

\section{Preliminaries}
\subsection{Hausdorff measures and dimensions}\label{Hausdorff}
Let $(X, d)$ be a metric space. For any subset $E\subset X$, $s\geq 0$ and $r>0$, the \emph{$s$-dimensional Hausdorff measure} of $E$ is defined by
\begin{equation*}
   \mathcal{H}^s(E):=\lim_{r\rightarrow 0}\inf\Big\{\sum_{i\geq 1}|U_i|^s: E\subset\bigcup_{i\geq 1}U_i\;\text{and}\;
   |U|_i\leq r\;\text{for all}\;i\geq 1\Big\},
\end{equation*}
where $|U|$ denotes the diameter of $U$. The \emph{Hausdorff dimension} of $E$ is defined by
\begin{equation*}
   \dim_{\rm H} E:=\inf\big\{s\geq 0: \mathcal{H}^s(E)=0\big\}=\sup\big\{s\geq 0: \mathcal{H}^s(E)=\infty\big\}.
\end{equation*}
Let $E$ be a totally bounded set in $X$, i.e., for any $r>0$, $E$ can be covered by finite number of balls of diameter $r$. Let $N_{r}(E)$ be the minimal number of balls of diameter $r$ to cover $E$. The \emph{upper and lower box dimension of $E$} are defined respectively by
\begin{equation}\label{box}
    \overline{\dim}_{\rm B} E=\limsup_{r\rightarrow 0}\frac{\log N_{r}(E)}{-\log r}\;\text{and}\;
    \underline{\dim}_{\rm B} E=\liminf_{r\rightarrow 0}\frac{\log N_{r}(E)}{-\log r}.
\end{equation}
If the limsup and liminf in \eqref{box} coincide, then the common value is defined as the \emph{box dimension of $E$}, denoted by $\dim_{\rm B} E$. It is well known that
$\dim_{\rm H} E\leq\underline{\dim}_{\rm B} E\leq\overline{\dim}_{\rm B} E$ holds for any totally bounded set, and all of them are bi-Lipschitz invariants.

Let $\mu$ be Borel probability measure on $X$, the Hausdorff dimension of $\mu$ if defined by
\begin{equation*}
   \dim_{\rm H}\mu=\inf\left\{\dim_{\rm H} E: E\;\text{is Borel}\;,\mu(E)=1\right\}.
\end{equation*}
We call $\mu$ a \emph{measure of maximal dimension on $X$} if $\dim_{\rm H}\mu=\dim_{\rm H}X$. See \cite{Fal90} and \cite{Mat95} for more details on Hausdorff measures and dimensions, and for \cite{FLR02} on dimension of measures.
The next two lemmas connect the Hausdorff measures and dimensions in two metric spaces.
\begin{lem}\label{keylemma1}\emph{(\cite[Lemma 2.4]{GP97})}
  Let $X$ and $Y$ be two compact metric spaces and $\pi: X\rightarrow Y$ be a continuous map satisfying a H\"{o}lder
condition with exponent $a>0$. Then for any $E\subset X$ and any $s\geq 0$, we have
\begin{equation*}
  \mathcal{H}^s(E)<\infty\; \Rightarrow \; \mathcal{H}^{s/a}(\pi(E))<\infty.
\end{equation*}
Thus, $$\dim_{\rm H}\pi(E)\leq a^{-1}\dim_{\rm H}E.$$ Moreover, for any compact subset $F\subset X$,
$$\overline{\dim}_{\rm B}\pi(F)\leq a^{-1}\overline{\dim}_{\rm B}F \;\;\text{and}\;\;\underline{\dim}_{\rm B}\pi(F)\leq a^{-1}\underline{\dim}_{\rm B}F.$$
\end{lem}

\begin{lem}\label{keylemma2}\emph{(\cite[Lemma 5.1]{GP97})}
  Let $X$ and $Y$ be two compact metric spaces and $\pi: X\rightarrow Y$ be a continuous map satisfying the following
property: there exists $b, c, N>0$ such that for all sufficiently small $\epsilon>0$ and any ball $B$ of radius $\epsilon$ in $Y$, the preimage $\pi^{-1}(B)$ can be covered by at most $N$ sets of diameter less than $c\epsilon^b$. Then for any set $E\subset Y$, we have that
\begin{equation*}
  \mathcal{H}^s(\pi^{-1}(E))>0\; \Rightarrow \; \mathcal{H}^{sb}(E)>0.
\end{equation*}
Thus, $$\dim_{\rm H} E\geq b\dim_{\rm H}\pi^{-1}(E).$$ Moreover, for any compact subset
$F\subset X$, $$\overline{\dim}_{\rm B}F\geq b^{-1}\overline{\dim}_{\rm B}\pi^{-1}(F)\;\;\text{and}\;\;
\underline{\dim}_{\rm B}F\geq b^{-1}\underline{\dim}_{\rm B}\pi^{-1}(F).$$
\end{lem}

\subsection{Notations and general results in dynamical systems}\label{setting}
We fix a finite $\Lambda=\{1,\cdots,m\}$, called an \emph{alphabet}, let $\Omega=\Lambda^{\mathbb{N}}$ be the product space endowed with the product topology and $\Lambda^*$ be the set of all finite words over $\Lambda$ (including empty word by convention). Denote the left shift map on $\Omega$ by $\sigma$. $\Omega$ is called the \emph{full shift}, and any $\sigma$-invariant and closed subset $\Sigma\subset\Omega$ is called a \emph{subshift}. For each $n\geq 1$, let $\Lambda^n$ be the set of all finite words of length $n$ over $\Lambda$, and $\Lambda^0$ denotes the set of empty word. The \emph{language} of $\Sigma$, denoted by $\mathcal{L}(\Sigma)$ and abbreviated as $\mathcal{L}$, is the set of all finite words appearing in $\Sigma$. Denote $\mathcal{L}_n=\mathcal{L}\cap\Lambda^n$. Given $\textbf{i}\in\mathcal{L}$, we write $|\textbf{i}|$ for the length of $\textbf{i}$. For a finite word $\textbf{i}=i_1\ldots i_n\in\mathcal{L}$ and each $1\leq k\leq n$, denote by $\textbf{i}|_1^k$ the prefix of $\textbf{i}$. Similarly, for an infinite word $w\in\Sigma$ and each $k\geq 1$, let $w|_1^k$ be the prefix of $w$. Throughout this paper, we shall use the bolded letters like $\textbf{i}, \textbf{j}, \textbf{k},\ldots$ to denote finite words in $\Lambda^{*}$. Write $\textbf{i}^{-}=\textbf{i}|_1^{|\textbf{i}|-1}$. We use $\textbf{i}\textbf{j}$ to denote the concatenation of $\textbf{i}$ and $\textbf{j}$, and $\textbf{i}^\infty$ denotes the concatenation of infinitely many copies of $\textbf{i}$. For each $\textbf{i}\in\mathcal{L}_n$, we use
\begin{equation*}
  [\textbf{i}]=\{w\in\Sigma: w|_1^n=\textbf{i}\}
\end{equation*}
to denote the \emph{$n$-th cylinder associated with $\textbf{i}$ in $\Sigma$}. It is known that all cylinders in $\Sigma$ are open and closed under the product topology (see \cite{LM95}). It is easy to see that each point in a cylinder is the center of that cylinder. Moreover, if the intersection of two cylinders in $\Sigma$ is nonempty, then one must be contained in the other.

Let $\varphi:\Sigma\rightarrow\mathbb{R}_{+}$ be a strictly positive continuous function, denote $S_0\varphi(w)=0$ and $S_n\varphi(w)=\sum_{i=0}^{n-1}\varphi(\sigma^i w)$ for each $n\geq 1$. Define
$$S_n^*\varphi(w)=\min_{w'\in [w|^n_1]}S_n\varphi(w'),$$
which depends only on the first $n$ coordinates of $w$. The \emph{$n$-th variation} of $\varphi$ is
\begin{equation*}
 \text{Var}_n\varphi=\sup\big\{|\varphi(w)-\varphi(w')|: w|_1^n=w'|_1^n\big\}.
\end{equation*}
Since $\varphi$ is uniformly continuous on $\Sigma$, we have $\lim_{n\rightarrow\infty}\text{Var}_n\varphi=0$. Thus, for any $w\in\Sigma$,
\begin{equation*}
   0\leq \frac{1}{n}\big(S_n\varphi(w)-S_n^*\varphi(w)\big)\leq\frac{1}{n}\sum_{i=1}^n\text{Var}_i\varphi\rightarrow 0.
\end{equation*}
Now we are ready to define a metric on $\Sigma$ by
\begin{equation}\label{metric}
   d_{\varphi}(w, w')=e^{-S_n^*\varphi(w)}, \; \forall w, w'\in\Sigma,
\end{equation}
where $n=\min\{i\geq 0: w_{i+1}\neq w'_{i+1}\}$. We note that, when $\varphi(w)\equiv\log m$, the metric $ d_{\varphi}$ is the one we used most often.
Since $\varphi$ is strictly positive,  $S_n^*\varphi(w)\rightarrow\infty$ for any $w$. It readily follows that
$d_{\varphi}$ induces the product topology on $\Sigma$. Denote by $|\cdot|_{\varphi}$ the diameter of a set with respect to $d_{\varphi}$,
then we have
\begin{equation*}
   \left|[w|^n_1]\right|_{\varphi}=e^{-S_n^*\varphi(w)}, \; \forall w\in\Sigma, \;
   n\geq 1.
\end{equation*}
For a Borel subset $E\subset\Sigma$, we define the $s$-dimensional Hausdorff measure, $\varphi$-Box dimension and $\varphi$-Hausdorff dimension of subsets of $\Sigma$ in the usual way, and denote them by $\mathcal{H}_{\varphi}^s(E)$, $\dim^\varphi_{\rm B}$ and $\dim^\varphi_{\rm H}$ respectively.

Let $C(\Sigma)$ be the space of continuous functions on $\Sigma$. Given $\phi\in C(\Sigma)$ and a collection of words $\mathcal{D}\subset\mathcal{L}$, we write $\mathcal{D}_n=\mathcal{D}\cap\Lambda^n$ for each $n\geq 1$ and consider the quantities
\begin{equation}\label{Z-n}
  Z_n(\sigma, \mathcal{D}, \phi):=\sum_{\textbf{i}\in\mathcal{D}_n}\sup_{w\in [\textbf{i}]}e^{S_n\phi(w)}.
\end{equation}
The \emph{upper capacity of $\phi$ on $\mathcal{D}$} is given by
\begin{equation}\label{uppercapacity}
  P(\sigma, \mathcal{D}, \phi)=\limsup_{n\rightarrow\infty}\frac{1}{n}\log Z_n(\sigma, \mathcal{D}, \phi).
\end{equation}
When $\mathcal{D}=\mathcal{L}$, the limit in right-hand side of \eqref{uppercapacity} exists due to the subadditivity:
\begin{equation*}
  \log Z_{m+n}(\sigma, \mathcal{L}, \phi)\leq \log Z_m(\sigma, \mathcal{L}, \phi)+\log Z_n(\sigma, \mathcal{L}, \phi),
\end{equation*}
thus we recover the standard definition of \emph{topological pressure}, and we write $P(\sigma, \phi)$ in place of $P(\sigma, \mathcal{L}, \phi)$. If we take $\phi(w)\equiv 0$, then the quantity defined by \eqref{uppercapacity} characterizes the exponential rate of growth of $\sharp\mathcal{D}_n$, we denote it by $h_{top}(\mathcal{D})$. If $\phi(w)\equiv 0$ and $\mathcal{D}=\mathcal{L}$, then we still have a limit in \eqref{uppercapacity}, which is the \emph{the topological entropy} of $(\Sigma, \sigma)$ (see \cite[Def. 4.1.1.] {LM95}), denoted by $h_{top}(\sigma, \Sigma)$. See \cite{Wa82} for more backgrounds on topological entropy and topological pressures.

Let $\mathcal{M}(\Sigma, \sigma)$ be the set of $\sigma$-invariant Borel probability measures on $\Sigma$, and let $h(\mu, \sigma)$ be the measure-theoretic entropy of $\mu\in\mathcal{M}(\Sigma, \sigma)$. The variational principle states that
\begin{equation*}
   P(\sigma, \phi)=\sup\left\{h(\mu, \sigma)+\int\phi d\mu: \mu\in\mathcal{M}(\Sigma, \sigma)\right\}.
\end{equation*}
A measure $\mu\in\mathcal{M}(\Sigma, \sigma)$ that attains the supremum is called the \emph{equilibrium state} for $\phi$. For an ergodic measure $\mu\in\mathcal{M}(\Sigma, \sigma)$, the $\varphi$-Hausdorff-dimension of $\mu$ is related to the measure-theoretic entropy by the follow lemma.
\begin{lem} \emph{(See \cite[Corollary 2.3]{GP97})}\label{Youngformula}
  Let $\mu$ be a $\sigma$-invariant and ergodic probability measure on $\Sigma$ and let $\varphi: \Sigma\rightarrow\mathbb{R}_{+}$ be a strictly positive continuous function. Then
 \begin{equation}\label{LY-formula}
   \dim^\varphi_{\rm H}\mu=\frac{h(\mu, \sigma)}{\int\varphi d\mu}.
 \end{equation}
\end{lem}

We also need Bowen's lemma on the topological entropy of quasi-regular points. Let $X$ be a compact metric space, then the set $\mathcal{M}(X)$ of all Borel probability measures on $X$ is a compact metrizable space under the weak topology. For a continuous map $T$ on $X$, let $\mathcal{M}(X, T)$ be the set of all $T$-invariant Borel probability measures on $X$ and $\mathcal{M}_e(X, T)$ be all $T$-invariant and ergodic measures. For $x\in X$, let $\mu_{n, x}=\frac{1}{n}\sum_{i=0}^{n-1}\delta_{T^i x}$ be the uniform measure supported on $\{T^i x\}_{i=0}^{n-1}$. Let $V_T(x)$ be the set of all limit points of $\{\mu_{n, x}\}_{n\geq 1}$ in $\mathcal{M}(X)$. Then $V_T(x)\subset\mathcal{M}(X, T)$. We say that $x$ is a \emph{$\mu$-generic point} for some $\mu\in\mathcal{M}(X)$ if $V_T(x)=\{\mu\}$.

\begin{thm}\label{Bowen}\emph{(\cite[Bowen's lemma]{Bow73})}
  Let $T: X\rightarrow X$ be a continuous map on a compact metric space. Set
\begin{equation*}
  QR(t)=\{x\in X: \exists\mu\in V_T(x)\;\text{with}\; h(\mu, T)\leq t\},
\end{equation*}
then $h_{top}(T, QR(t))\leq t$. Here $h_{top}(T, \cdot)$ means the topology entropy of subsets in Bowen's sense.
\end{thm}

\subsection{Specification properties and measures of maximal dimension}
\label{specification}
The specification property, first introduced by Bowen \cite{Bow74} to study the existence and uniqueness of equilibrium states for some expansive dynamical systems, could be viewed as a quantitative version of topological mixing. It takes a simple form for subshifts.
\begin{defn}
   Let $\Sigma$ be a subshift and $\mathcal{L}$ be its language.
\begin{enumerate}
\item A subcollection $\mathcal{G}\subseteq\mathcal{L}$ is said to satisfy the specification property if there exists $\tau\in\mathbb{N}$ such that for all $k\geq 1$ and $\textbf{i}_1, \ldots, \textbf{i}_k\in\mathcal{G}$, there exist  $\textbf{j}_1,\ldots, \textbf{j}_{k-1}\in\mathcal{L}_{\tau}$ such that $\textbf{i}_1\textbf{j}_1\textbf{i}_2\textbf{j}_2\cdots\textbf{j}_{k-1}\textbf{i}_k\in\mathcal{L}$. The constant $\tau$ is called the gap size.

%\item Furthermore, $\mathcal{G}$ is said to satisfy the periodic specification property if there exists $\tau\in\mathbb{N}$ such that for every $\textbf{i}, \textbf{k}\in\mathcal{G}$, we can find $\textbf{j}\in\mathcal{L}_{\tau}$ such $\textbf{i}\textbf{j}\textbf{k}\in\mathcal{L}$,
%and moreover $w=(\textbf{i}\textbf{j}\textbf{k})^{\infty}$ is a periodic point in $\Sigma$.

\item $\Sigma$ is said to have the specification property if so does its language $\mathcal{L}$.
\end{enumerate}
\end{defn}

For subshifts with specification property, Gatzouras and Peres \cite{GP97} proved the following result concerning measures of maximal dimension.
\begin{thm}[\cite{GP97}, Theorem 2.1]\label{GP}
  Let $\Sigma$ be a subshift and $\varphi: \Sigma\rightarrow\mathbb{R}_+$ be a strictly positive continuous function. Endow $\Sigma$ with the metric $d_{\varphi}$ as above. Then
  \begin{enumerate}
  \item there exists a $\sigma$-invariant ergodic probability measure $\mu$ on $\Sigma$ with $\dim_{\rm H}^{\varphi}\mu=\dim_{\rm H}^{\varphi}\Sigma$;
  \item the $\varphi$-Hausdorff dimension $\gamma=\dim_{\rm H}^{\varphi}\Sigma$ and $\varphi$-Box dimension $\dim_{\rm B}^{\varphi}\Sigma$ coincide. Moreover, $\gamma$ satisfies the equation $P(\sigma, -\gamma\varphi)=0$;
  \item if $\Sigma$ satisfies the specification property and $\varphi$ satisfies the Dini condition $\sum_{n\geq 1}\text{Var}_n\varphi<\infty$, then $$0<\mathcal{H}_{\varphi}^{\gamma}(\Sigma)<\infty$$ and the measure in (1) is unique among all ergodic measures. Moreover, $\mu$ has the following Gibbs property: there exists a constant $c>1$ such that for every $n\geq 1$ and $w\in\Sigma$, we have
  \begin{equation*}
      c^{-1} e^{-\gamma S_n\varphi(w)}\leq\mu([w|^n_1])\leq c e^{-\gamma S_n\varphi(w)}.
  \end{equation*}
\end{enumerate}
\end{thm}

\begin{rem}\label{LYtopFormula}
By a property of topological pressure \emph{(}see \cite[Theorem 9.7]{Wa82}\emph{)},
\begin{equation*}
   h_{top}(\sigma, \Sigma)-\gamma\max_{w\in\Sigma}\varphi(w)\leq P(\sigma, -\gamma\varphi)\leq h_{top}(\sigma, \Sigma)-\gamma\min_{w\in\Sigma}\varphi(w).
\end{equation*}
Thus, Theorem \ref{GP} (2) implies that for any subshift $\Sigma$ and any strictly positive continuous function $\varphi: \Sigma\rightarrow\mathbb{R}_+$, we have
\begin{equation}\label{LY}
   \frac{h_{top}(\sigma, \Sigma)}{\max_{w\in\Sigma}\varphi(w)}\leq\dim_{\rm
 H}^{\varphi}\Sigma\leq\frac{h_{top}(\sigma, \Sigma)}{\min_{w\in\Sigma}\varphi(w)}
\end{equation}
\end{rem}

In recent years, various weak versions of specification property were introduced to study the uniqueness of equilibrium states of dynamical systems along Bowen's approach, see the survey paper \cite{CT21} and references therein for more details. In this paper, we adopt a version of weak specification property from \cite[Section 2]{CT13}, which is satisfied for a large class of symbolic dynamics, including the $\beta$-shifts, S-gap shifts and more general coded systems (we refer to \cite{CT12, CT21} for more examples). We assume that there exist collections of words $\mathcal{C}^p, \mathcal{G}, \mathcal{C}^s\subseteq\mathcal{L}$ such that $\mathcal{L}=\mathcal{C}^p\mathcal{G} \mathcal{C}^s$, that is, every word in $\mathcal{L}$ can be written as a concatenation of words in $\mathcal{C}^p, \mathcal{G}$ and $ \mathcal{C}^s$. For each $M\geq 0$, write
  \begin{equation*}
    \mathcal{G}(M):=\left\{\textbf{i}\textbf{j}\textbf{k}\in\mathcal{L}: \textbf{i}\in\mathcal{C}^p, \textbf{j}\in\mathcal{G}, \textbf{k}\in\mathcal{C}^s, |\textbf{i}|\leq M, |\textbf{k}|\leq M\right\}.
  \end{equation*}
Note that $\bigcup_M \mathcal{G}(M)=\mathcal{L}$, thus
$\{\mathcal{G}(M)\}_{M\geq 0}$ gives a filtration of $\mathcal{L}$. In this paper, for the sake of our use, we will extend Theorem \ref{GP} by weakening the specification property of $\Sigma$. Since the first two items in Theorem \ref{GP} hold for any subshift and any strictly positive continuous function $\varphi$, we omit them in the following result.
\begin{thm}\label{extendGP}
    Let $\Sigma$ be a subshift and $\varphi$ be a strictly positive continuous function on $\Sigma$. Write $\gamma=\dim_{\rm H}^{\varphi}\Sigma=\dim_{\rm B}^{\varphi}\Sigma$.
    Suppose that there exists a decomposition $\mathcal{L}=\mathcal{C}^p\mathcal{G} \mathcal{C}^s$ satisfying:
    \begin{enumerate}[(i)]
    \item $\mathcal{G}$ has the specification property;
    \item $\varphi$ satisfies the Dini condition $\sum_{n\geq 1}\text{Var}_n\varphi<\infty$;
    \item the pressure gap condition, that is, $\sum_{n\geq 1} Z_n(\sigma, \mathcal{C}^p\cup\mathcal{C}^s, -\gamma\varphi)<\infty$.
\end{enumerate}
Then there exists a unique $\sigma$-invariant ergodic probability measure $\mu$ with $\dim_{\rm H}^{\varphi}\mu=\dim_{\rm H}^{\varphi}\Sigma$. Moreover, $\mu$ has the the following weak Gibbs property: there exist constants $c_M, c>0$ such that for every $n\geq 1$ and $\textbf{i}\in\mathcal{G}(M)_n$, we have
\begin{equation}\label{gibbs}
    c_M e^{-\gamma S_n\varphi(w)}\leq\mu([\textbf{i}])\leq c e^{-\gamma S_n\varphi(w)}, \forall w\in [\textbf{i}].
\end{equation}
\end{thm}

\begin{rem}\label{Uniformgibbs}
    Note that $\mathcal{G}=\mathcal{G}(0)$, thus \eqref{gibbs} tells that there exists $c>0$ such that for every $n\geq 1$ and $\textbf{i}\in\mathcal{G}_n$,
\begin{equation}
    c^{-1} e^{-\gamma S_n\varphi(w)}\leq\mu([\textbf{i}])\leq c e^{-\gamma S_n\varphi(w)}, \forall w\in [\textbf{i}].
\end{equation}
In the proof of Theorem \ref{extendGP}, we have the following uniform counting result \emph{(}see \cite[Proposition 5.3]{CT13}\emph{)}: for any potential function $\phi\in C(\Sigma)$ satisfying Dini condition, there exists $C>1$ such that
\begin{equation}\label{Uniformcount}
     e^{n P(\sigma, \phi)}\leq Z_n(\sigma, \mathcal{L}, \phi) \leq C e^{n P(\sigma, \phi)}, \;\forall\; n\geq 1.
\end{equation}
In particular, we have
\begin{equation}\label{Uniformcount2}
    1\leq Z_n(\sigma, \mathcal{L}, -\gamma\varphi) \leq C, \;\forall\; n\geq 1.
\end{equation}
\end{rem}

For the proof of Theorem \ref{extendGP}, the uniqueness and weak Gibbs property follows from \cite[Theorem C]{CT13}, where Condition (i) was replaced by a stronger assumption that $\mathcal{G}(M)$ satisfies the specification property for every $M\geq 0$. It was extended recently in \cite{PYY22} by only assuming the specification property of $\mathcal{G}$. Although the results in \cite{PYY22} were proved for flows rather than symbolic dynamics, the arguments there still work for subshifts, see \cite{Cli22} for an excellent exposition about this.

%\begin{prop}[\cite{CT13}, Proposition 5.3]\label{uniformcount}
%   There exists $C_1>0$ such that for all continuous function $\phi\in C(\Sigma)$ and
% $n\geq 1$, we have
%\begin{equation*}
%    e^{nP(\phi)}\leq Z_n(\mathcal{L}, \phi)\leq C_1 e^{nP(\phi)}.
%\end{equation*}
%\end{prop}
% $\text{Per}_n(\Sigma)=\{w\in\Sigma: \sigma^n(w)=w\}$

\section{Proof of Theorem \ref{int-entropy}, \ref{mainthm} and \ref{intermediate-dimension}}
\subsection{Proof of Theorem \ref{int-entropy}}
 By the assumptions, the language $\mathcal{L}$ of $\Sigma$ admits a decomposition $\mathcal{L}=\mathcal{C}^p\mathcal{G}\mathcal{C}^s$, where $\mathcal{G}$ has the specification property with gap size $\tau\in\mathbb{N}$ and $h_{top}(\mathcal{C}^p\cup\mathcal{C}^s)<h_{top}(\sigma, \Sigma)$. The next lemma reveals that, under our assumptions, the topological entropy of $\Sigma$ is determined by the exponential growth rate of $\mathcal{G}$.
\begin{lem}\label{goodcore}
    Assume that $(\Sigma, \sigma)$ and the subcollections $\mathcal{C}^p, \mathcal{G}, \mathcal{C}^s$ in the decomposition $\mathcal{L}=\mathcal{C}^p\mathcal{G}\mathcal{C}^s$ satisfy the conditions in Theorem \ref{int-entropy}, then we have
\begin{equation}\label{G-rate}
   h_{top}(\mathcal{G}):=\limsup_{n\rightarrow\infty}\frac{1}{n}\log\sharp\mathcal{G}_n=h_{top}(\sigma, \Sigma)
\end{equation}
\end{lem}
\begin{proof}
For convenience, we abbreviate $h_{top}(\sigma, \Sigma)$ as $h_0$ in the proof. Assume the assertion is not true, that is, $h_{top}(\mathcal{G})< h_0$. Then $h_{top}(\mathcal{G})< h_0-\epsilon_1$ holds for any sufficiently small $\epsilon_1>0$. Thus, there exists a constant $C_1>0$ such that for every $n\geq 1$,
\begin{equation}\label{CG1}
  \sharp\mathcal{G}_n\leq C_1 e^{n(h_0-\epsilon_1)}.
\end{equation}
Similarly, the entropy gap condition implies that, for any sufficiently small $\epsilon_2>0$, there exists a constant $C_2>0$ such that for every $n\geq 1$,
\begin{equation}\label{CPCS}
 \sharp(\mathcal{C}^p\cup\mathcal{C}^s)_n\leq C_2 e^{n(h_0-\epsilon_2)}.
\end{equation}
Since each $\textbf{i}\in\mathcal{L}$ has the decomposition $\textbf{i}=\textbf{i}^p\textbf{i}^g\textbf{i}^s$ with $\textbf{i}^p\in\mathcal{C}^p, \textbf{i}^g \in\mathcal{G}, \textbf{i}^s\in\mathcal{C}^s$. Thus,
$$\mathcal{L}_n=\bigcup_{0\leq k\leq n}\big\{\textbf{i}^p\textbf{i}^g\textbf{i}^s: \textbf{i}^p\in\mathcal{C}^p, \textbf{i}^s\in\mathcal{C}^s, \textbf{i}^g \in\mathcal{G}_{n-k}\big\}.$$
Fix $\epsilon_1$ and $\epsilon_2$ in \eqref{CG1} and \eqref{CPCS}, we have
\begin{align}\label{LG}
\sharp\mathcal{L}_n\leq
\sum_{0\leq k\leq n\atop 0\leq j\leq k} \sharp\mathcal{C}^p_j\cdot\sharp\mathcal{G}_{n-k}\cdot\sharp\mathcal{C}^s_{k-j}
\leq C_1 C_2^2  \sum_{0\leq k\leq n} e^{(n-k)(h_0-\epsilon_1)} e^{k(h_0-\epsilon_2)}\lesssim e^{n(h_0-\epsilon_1)},
\end{align}
which contradicts with the fact that $\sharp\mathcal{L}_n\geq e^{n h_{top}(\sigma, \Sigma)}$ (see \cite[Lemma 5.1]{CT12}).
\end{proof}

\noindent\emph{Proof of Theorem \ref{int-entropy}}
We deal with the case $h=h_{top}(\sigma, \Sigma)$ and $h\in [0, h_{top}(\sigma, \Sigma))$ separately. However, the constructions in the two cases are quite the same, i.e., taking the limit of a sequence of finite words in $\mathcal{L}$. For a point $w\in\Sigma$, let $p_w(n)$ be the number of different words of length $n$ appearing in $w$. It is easily seen that $\sharp\mathcal{L}_n(\overline{\mathcal{O}_{\sigma}(w)})=p_w(n)$.

%$\bullet$ For the case $h=0$, since $h_{top}(\sigma, \Sigma)>0$, by Lemma \ref{goodcore}, we know $\mathcal{G}\neq\emptyset$. Take $\textbf{u}\in\mathcal{G}_1$. By the specification property of $\mathcal{G}$, there exists a sequence of finite words $\{\textbf{v}_k\}_{k\geq 1}\subset\mathcal{L}_{\tau}$
%such that
%$$\textbf{u}\textbf{v}_1\textbf{u}\textbf{v}_2\cdots\textbf{v}_{n-1}\textbf{u}\in\mathcal{L} \;\;\text{for all}\;\; n\geq 1.$$
%Let $w=\textbf{u}\textbf{v}_1\textbf{u}\textbf{v}_2\cdots\textbf{v}_{n-1}\textbf{u}\cdots$, then $w\in\Sigma$. Since the infinite sequence $\{\textbf{v}_k\}_{k\geq 1}$ take values in a finite set, it is eventually periodic by pigeonhole principle. Thus, $w$ is eventually periodic of period at most $(\tau+1)m^{\tau}$, it is easy to see that $p_w(n)\leq 2(\tau+1)m^{\tau}$ holds for any $n\geq (\tau+1)m^{\tau}$. As a result, we have $h_{top}(\sigma, \overline{\mathcal{O}_{\sigma}(w)})=0$.

\smallskip
$\bullet$ For the case $h=h_{top}(\sigma, \Sigma)$, firstly we arrange all elements of $\mathcal{G}$ by lexicographic order and write them as a sequence $\{\textbf{u}_k\}_{k\geq 1}$. By the specification property of $\mathcal{G}$, there exists a sequence of finite words $\{\textbf{v}_k\}_{k\geq 1}\subset\mathcal{L}_{\tau}$
such that
$$\textbf{u}_1\textbf{v}_1\textbf{u}_2\textbf{v}_2\cdots\textbf{v}_{n-1}\textbf{u}_n\in\mathcal{L} \;\;\text{for all}\;\; n\geq 1.$$
Let $w=\textbf{u}_1\textbf{v}_1\textbf{u}_2\textbf{v}_2\cdots\textbf{v}_{n-1}\textbf{u}_n\cdots$, then $w\in\Sigma$. Since all elements of $\mathcal{G}$ appear in $w$, we have
$$\sharp\mathcal{G}_n\leq p_w(n)\leq\sharp\mathcal{L}_n.$$
Then it follows from Lemma \ref{goodcore} that
$h_{top}(\sigma, \overline{\mathcal{O}_{\sigma}(w)})=h_{top}(\sigma, \Sigma)$.

\smallskip
$\bullet$ For the case $0\leq h<h_{top}(\sigma, \Sigma)$, we modify the construction of Grillenberger \cite{Gri72} where $\Sigma$ is a full shift (see also \cite{Bru98}). The original construction in \cite{Gri72} is based on free concatenation of finite words, which is impossible for general subshifts. Instead, by assuming the specification property $\mathcal{G}$, we will construct the required orbit by carefully gluing some initial chosen words in $\mathcal{G}$. The construction is divided into several steps:

\begin{enumerate}
  \item Step 1.
  Let $l_1, n_1$ be two sufficiently large integers satisfying $(l_1+\tau) h+2\leq\log n_1<(l_1+\tau) h+3$. Take a collection of $n_1$ words in $\mathcal{G}_{l_1}$, write it as $\mathcal{E}_1=\{\textbf{i}_1,\cdots,\textbf{i}_{n_1}\}$ and fix it. We call words in $\mathcal{E}_1$ the \emph{brick words}. By the specification property of $\mathcal{G}$, for
  each infinite configuration $(\textbf{i}_{j_1},\textbf{i}_{j_2},\cdots, \textbf{i}_{j_k},\cdots)\in\mathcal{E}_1^{\mathbb{N}}$ there exists a sequence $\{\textbf{v}_j\}_{j\geq 1}\subset\mathcal{L}_{\tau}$ such that
\begin{equation}\label{glue}
  \textbf{i}_{j_1}\textbf{v}_1\textbf{i}_{j_2}\textbf{v}_2\cdots\textbf{v}_{k-1}\textbf{i}_{j_k}\in\mathcal{L}\;\;\text{for all}\;\; k\geq 1.
\end{equation}
Note that there might be multiple choices of $\{\textbf{v}_j\}_{j\geq 1}$ satisfying \eqref{glue}, we shall fix one of them for
each choice $\mathcal{E}_1$ of initial brick words
and each infinite configuration $(\textbf{i}_{j_1},\textbf{i}_{j_2},\cdots, \textbf{i}_{j_k},\cdots)\in\mathcal{E}_1^{\mathbb{N}}$, and call it the sequence of \emph{gluing words}. In addition, given the $n_1$ different brick words in $\mathcal{E}_1$, there are $n_1!$ different permutations of them, each of which will produce a word of length $l_{2}=l_1 n_1+ (n_1-1)\tau$, which takes the form $\underbrace{\star\ast\star\ast\cdots\star\ast\star}_{n_1}$ with $\star\in\mathcal{E}_1$ and $\ast$ being the gluing words. Let $\textbf{u}_1$ be one of those words in $\mathcal{E}_1$.

  \smallskip
  \item From Step $k$ to Step $k+1$. Assume that at Step $k$, we have obtained $n_k$ different words in $\mathcal{L}_{l_k}$, write this collection as $\mathcal{E}_k$. Note that the elements of
  $\mathcal{E}_k$ are concatenations of those brick words in $\mathcal{E}_1$ and the gluing words in $\mathcal{L}_{\tau}$, which are of the form $\star\ast\star\ast\cdots\star\ast\star$ with $\star\in\mathcal{E}_1$ and $\ast$ being the gluing words. Thus, there are $n_k!$ different permutations of them, each of which will produce a word of length $l_{k+1}=l_k n_k+(n_k-1)\tau$ if we glue those brick words in them together. We denote the collection of such $n_k!$ different words of length $l_{k+1}$ by $\widetilde{\mathcal{E}_k}$. Let $\mathcal{E}_{k+1}$ be a subcollection of $\widetilde{\mathcal{E}_k}$ whose cardinality $n_{k+1}=\sharp\mathcal{E}_{k+1}$ satisfies $$(l_{k+1}+\tau) h+2\leq\log n_{k+1}<(l_{k+1}+\tau) h+3.$$ The choice is possible, since by induction, we have $n_k>4$ and by Stirling's formula,
  \begin{equation}\label{const}
     \log(n_k!)\geq n_k(\log n_k-1)\geq n_k(l_{k}\alpha+\tau\alpha+1)>l_{k+1} h+2.
  \end{equation}
      Let $\textbf{u}_{k+1}$ be one of $\mathcal{E}_{k+1}$.

\smallskip
\item Remove the gluing words in the configuration $(\textbf{u}_1, \textbf{u}_2, \cdots)$, we obtain an infinite configuration $(\textbf{i}_{j_1},\textbf{i}_{j_2},\cdots, \textbf{i}_{j_k},\cdots)$ of brick words. According to the discussion in Step 1, there exists a sequence $\{\textbf{v}_j\}_{j\geq 1}\subset\mathcal{L}_{\tau}$ such that
$\textbf{i}_{j_1}\textbf{v}_1\textbf{i}_{j_2}\textbf{v}_2\cdots\textbf{v}_{k-1}\textbf{i}_{j_k}\in\mathcal{L}\;\;\text{for all}\;\; k\geq 1.$
Let $w=\textbf{i}_{j_1}\textbf{v}_1\textbf{i}_{j_2}\textbf{v}_2\cdots\textbf{v}_{k-1}\textbf{i}_{j_k}\cdots$, then $w\in\Sigma$.
\end{enumerate}

%We claim that $\overline{\mathcal{O}_{\sigma}(w)}$ is minimal. Indeed, by the construction, each $l_i$-block returns within $l_{i+2}$ iterations of $\sigma$. Therefore, $(\overline{\mathcal{O}_{\sigma}(w)}, \sigma)$ is uniformly recurrent, which means that $(\overline{\mathcal{O}_{\sigma}(w)}, \sigma)$ is a minimal subshift (see \cite{Lot02}). Recall that a sequence $w\in\Sigma_2$ is said to be uniformly recurrent if for every factor $B$ of $w$, there exists an integer $n_B$ such that $B$ appears in every block of length $n_B$.

Next we show that $h_{top}(\sigma, \overline{\mathcal{O}_{\sigma}(w)})=h$. Since at Step $k$ of the construction, there are $n_k$ different words of length $l_k$, thus $p_w(l_k)\geq n_k$. Then we have
\begin{equation*}
  h_{top}(\sigma, \overline{\mathcal{O}_{\sigma}(w)})=\lim_{n\rightarrow\infty}\frac{1}{n}\log \sharp\mathcal{L}_n(\overline{\mathcal{O}_{\sigma}(w)})=\lim_{n\rightarrow\infty}\frac{1}{n}\log p_w(n)\geq\lim\limits_{k\rightarrow\infty}\frac{1}{l_k}\log n_k\geq h.
\end{equation*}
On the other hand, for every factor $\textbf{u}$ of $w$ with length $l_{k+1}$, there exists $\textbf{v}=\textbf{j}_1\textbf{k}_1 \cdots \textbf{j}_{n_k}\textbf{k}_{n_k} \textbf{j}_{n_k+1}$ with $\textbf{j}_i$ sharing the same configuration of brick words as an element of $\mathcal{E}_k$ and $\textbf{k}_i\in\mathcal{L}_{\tau}$ such that $\textbf{u}$ is a factor of $\textbf{v}$. Therefore,
\begin{equation*}
  p_w(l_{k+1})\leq (l_k+2\tau) n_k^{n_k+1}m^{\tau\sum_{j=1}^{k-1}n_1\cdots n_j + \tau n_k}\leq 2 l_{k+1}n_k^{n_k}m^{\tau\sum_{j=1}^{k-1}n_1\cdots n_j + \tau n_k}.
\end{equation*}
Since both $\{n_k\}_{k\geq 1}$ and $\{l_k\}_{k\geq 1}$ are hyperexponential sequences tending to infinity, it follows that
\begin{equation*}
 \begin{split}
  h_{top}(\sigma, \overline{\mathcal{O}_{\sigma}(w)})=&\lim\limits_{n\rightarrow\infty}\frac{1}{n}\log p_w(n)
  \leq\lim\limits_{k\rightarrow\infty}\frac{\log (2 l_{k+1}n_k^{n_k}m^{\tau\sum_{j=1}^{k-1}n_1\cdots n_j + \tau n_k})}{l_{k+1}}\\ &\leq \lim\limits_{k\rightarrow\infty}\frac{\log 2 l_{k+1}}{l_{k+1}}+\lim\limits_{k\rightarrow\infty}\frac{\log n_k}{l_k}+\lim\limits_{k\rightarrow\infty}\frac{(\sum_{j=1}^{k-1}n_1\cdots n_j + n_k)\tau\log m}{l_{k+1}}\leq h.
 \end{split}
\end{equation*}
Therefore, $h_{top}(\sigma, \overline{\mathcal{O}_{\sigma}(w)})=h$.
\qed

\subsection{Proof of Theorem \ref{mainthm} modulo Theorem \ref{intermediate-dimension}}
\subsubsection{The upper bound.}
% We use Bowen's lemma on the topological entropy of the set of quasi-regular points
% (Theorem 2.6) and the relationship between entropy and dimension (Lemma 2.2).
Write $c^{*}=e^{\max_{w\in\Sigma}\varphi(w)}$. Fix an arbitrarily small positive number $\rho$, let
\begin{equation}\label{Lambda-rho}
    \Lambda_{\rho}:=\left\{\textbf{i}\in\mathcal{L}: \rho\leq |[\textbf{i}]|_{\varphi}<c^{*}\rho\right\}.
\end{equation}
We claim $\Sigma$ is contained in the union of the
cylinders in $\{[\textbf{i}]: \textbf{i}\in\Lambda_{\rho}\}$. Indeed, for any $w\in\Sigma$, let $k$ be the largest integer satisfying $|[w|_1^k]|_{\varphi}\geq\rho$, then we must have $|[w|_1^k]|_{\varphi}<c^{*}\rho$. If not, then $|[w|_1^k]|_{\varphi}\geq c^{*}\rho$ would imply $|[w|_1^{k+1}]|_{\varphi}\geq\rho$, which contracts with the maximality of $k$. Therefore, any $w\in\Sigma$ must belong to one cylinder in
$\{[\textbf{i}]: \textbf{i}\in\Lambda_{\rho}\}$. Note that if the intersection of two cylinders in $\{[\textbf{i}]: \textbf{i}\in\Lambda_{\rho}\}$ is nonempty, then one must be contained in the other. Therefore, we can remove those
$\textbf{i}\in\Lambda_{\rho}$ with $\textbf{i}^-\in\Lambda_{\rho}$ so that the remaining subcollection of cylinders $\{[\textbf{i}]: \textbf{i}\in\Lambda_{\rho}\}$ covers $\Sigma$ and the cylinders in the subcollection are pairwise disjoint. By abuse of notation, we still use $\Lambda_{\rho}$ to denote the resulted subcollection.

Since the collection $\Lambda_{\rho}$ is finite, we write it as
$\Lambda_{\rho}=\{\textbf{i}_1,\cdots, \textbf{i}_N\}$ with
$|\textbf{i}_i|=k_i$ for each $1\leq i\leq N$ and $k_1\geq\ldots\geq k_N$. We claim that
$\Sigma\subset\Lambda_{\rho}^{\mathbb{N}}$. Indeed, since $\{[\textbf{i}]: \textbf{i}\in\Lambda_{\rho}\}$ covers $\Sigma$, for any $w\in\Sigma$, there exist $\textbf{i}\in\Lambda_{\rho}$ such that $w\in [\textbf{i}]$. Since $\Sigma$ is $\sigma$-invariant, we have $\sigma^{|\textbf{i}|}(w)\in\Sigma$, which again implies that $\sigma^{|\textbf{i}|}(w)$ belongs to a cylinder in $\{[\textbf{i}]: \textbf{i}\in\Lambda_{\rho}\}$. Repeating this process, we shall see that $\Sigma\subset\Lambda_{\rho}^{\mathbb{N}}$. Remark that $\Lambda_{\rho}^{\mathbb{N}}\subset\Omega=\Lambda^{\mathbb{N}}$ is
a subshift, and it can also be regarded as a full shift over the alphabet $\Lambda_{\rho}$.

We now equip $\Lambda_{\rho}^{\mathbb{N}}$ a new metric $d_{\varphi}^{\rho}$, which depends on $\varphi$ and $\rho$. For any two sequences $u, u'\in\Lambda_{\rho}^{\mathbb{N}}$ with
$u=\textbf{u}_1\cdots \textbf{u}_k \textbf{u}_{k+1}\cdots$ and $u'=\textbf{u}_1\cdots \textbf{u}_k \textbf{u}'_{k+1}\cdots$, where $\textbf{u}_i, \textbf{u}'_i\in\Lambda_{\rho}, \textbf{u}_{k+1}\neq \textbf{u}'_{k+1}$ and the length of $\textbf{u}_i$ is $n_i$ for each $1\leq i\leq k$, we define
\begin{equation}\label{two_metric}
   d_{\varphi}^{\rho}(u, u'):=
\exp(-S_{n_1}^*\varphi(u)-S_{n_2}^*\varphi(\sigma^{n_1}u)- \cdots -S_{n_k}^*\varphi(\sigma^{n_1+\cdots+n_{k-1}}u)).
\end{equation}
Then $d_{\varphi}^{\rho}$ defines a metric on the symbolic space $\Lambda_{\rho}^{\mathbb{N}}$.
It can be easily derived that there exists $c_{\rho}>1$ such that
\begin{equation}\label{Lip}
  c_{\rho}^{-1}d_{\varphi}^{\rho}(u, u')\leq d_{\varphi}(u, u')\leq c_{\rho}e^{\sum_{i=1}^n\text{Var}_i\varphi}d_{\varphi}^{\rho}(u, u')
\end{equation}
for any $u, u'\in\Lambda_{\rho}^{\mathbb{N}}$, here $n=\sum_{i=1}^k n_i$. Let $\dim_{\rm H}^{\varphi, \rho}$ denote the Hausdorff dimension with respect to the metric $d_{\varphi}^{\rho}$. Then for any subset $A\subset\Sigma$, \eqref{Lip} implies that $\dim_{\rm H}^{\varphi, \rho} A=\dim_{\rm H}^{\varphi} A$.

Let $\sigma_{\rho}$ be the left shift map on the full shift $\Lambda_{\rho}^{\mathbb{N}}$ over $\Lambda_{\rho}$. Then for each $w\in\Sigma\subset\Lambda_{\rho}^{\mathbb{N}}$, we have $\mathcal{O}_{\sigma_{\rho}}(w)\subset\mathcal{O}_{\sigma}(w)$. Moreover,
\begin{equation}\label{suborbit}
  \overline{\mathcal{O}_{\sigma_{\rho}}(w)}\subset\overline{\mathcal{O}_{\sigma}(w)}\subset
\bigcup_{i=0}^{k_1}\sigma^i(\overline{\mathcal{O}_{\sigma_{\rho}}(w)}),
\end{equation}
here we recall that $k_1$ is the length of the longest word in $\Lambda_{\rho}$.
Since $\sigma$ is bi-Lipschitz on the metric space $(\Sigma, d_{\varphi})$, combining \eqref{Lip} and \eqref{suborbit} , we have
\begin{equation*}
 \dim_{\rm H}^{\varphi}\overline{\mathcal{O}_{\sigma}(w)}=\dim_{\rm H}^{\varphi}\overline{\mathcal{O}_{\sigma_{\rho}}(w)}
=\dim_{\rm H}^{\varphi, \rho}\overline{\mathcal{O}_{\sigma_{\rho}}(w)}
\end{equation*}
for each $w\in\Sigma$. As a result,
\begin{equation*}
  E_{\sigma}(\alpha)=\left\{w\in\Sigma: \dim_{\rm H}^{\varphi, \rho}\overline{\mathcal{O}_{\sigma_{\rho}}(w)}=\alpha\right\}.
\end{equation*}
Let
\begin{equation*}
  D_{\alpha}=\left\{w\in\Sigma: \dim_{\rm H}^{\varphi, \rho}\overline{\mathcal{O}_{\sigma_{\rho}}(w)}\leq\alpha\right\}.
\end{equation*}
Applying Remark \ref{LYtopFormula} to the subshift $\big(\overline{\mathcal{O}_{\sigma_{\rho}}(w)}, \sigma_{\rho}\big)$, by \eqref{two_metric} and the definition of $\Lambda_{\rho}$, we have
\begin{equation*}
\frac{h_{top}\big(\sigma_{\rho}, \overline{\mathcal{O}_{\sigma_{\rho}}(w)}\big)}{-\log\rho}\leq \dim_{\rm H}^{\varphi, \rho}\overline{\mathcal{O}_{\sigma_{\rho}}(w)}
\leq\frac{h_{top}\big(\sigma_{\rho}, \overline{\mathcal{O}_{\sigma_{\rho}}(w)}\big)}{-\log\rho-\log c^{*}}
\end{equation*}
Applying Theorem \ref{Bowen} to $\sigma_{\rho}$ implies
\begin{equation*}
  h_{top}(\sigma_\rho, D_{\alpha})\leq h_{top}\Big(\sigma_{\rho}, \big\{w\in\Sigma: h_{top}\big(\sigma_{\rho}, \overline{\mathcal{O}_{\sigma_{\rho}}(w)}\big)\leq -\alpha\log\rho\big\}\Big)\leq-\alpha\log\rho,
\end{equation*}
here $h_{top}(\sigma_{\rho}, \cdot)$ denotes the topological entropy with respect to $\sigma_{\rho}$ in Bowen's sense. Therefore,
\begin{equation*}
  \dim_{\rm H}^{\varphi} E_{\sigma}(\alpha)\leq \dim_{\rm H}^{\varphi}D_{\alpha}=\dim_{\rm H}^{\varphi, \rho} D_{\alpha}\leq \frac{h_{top}(\sigma_{\rho}, D_{\alpha})}{-\log\rho-\log c^{*}}\leq \frac{-\alpha\log\rho}{-\log\rho-\log c^{*}}.
\end{equation*}
Letting $\rho\rightarrow 0$, we obtain the upper bound.

\subsubsection{The lower bound}

% {\color{blue} The idea is essentially similar to that of Theorem 1.1. For each intermediate value $\alpha$, we start by choosing a finite set of finite words, then concatenate them repeatedly to get a subshift of dimension $\alpha$, in which almost every orbit closure has the prescribed dimension. To overcome the difficulty caused by non-constant potential function $\varphi$, we use ergodic theoretic method.}
% By Theorem 2.1 (2), for each subshift $F$ of $\Sigma$,  we have $\dim^\varphi_{\rm B} F=\dim^\varphi_{\rm H}F$.
By Theorem \ref{intermediate-dimension}, for any $\alpha\in [0, \gamma]$, there exists a closed invariant subset $X$ with a $\sigma$-invariant measure $\nu$ supported on it and satisfying $\dim_{\rm H}^{\varphi}X=\dim_{\rm H}^{\varphi}\nu$. By passing to ergodic decomposition, we may assume $\nu$ is ergodic.
Since $\nu$-almost every $x$ is generic with respect to $\nu$, therefore, for $\nu$-a.e. $x$,
the orbit of $x$ is dense in the compact subset $X$, thus $\dim_{\rm H}^{\varphi}\overline{\mathcal{O}_{\sigma}(x)}=\dim_{\rm H}^{\varphi}X=\alpha$. From this we infer that the set $G_{\nu}$ of generic points of $\nu$ consists of a subset of $E_{\sigma}^{\varphi}(\alpha)$. Since $\dim_{\rm H}^{\varphi}G_{\nu}=\dim_{\rm H}^{\varphi}X=\alpha$, we have $\dim_{\rm H}^{\varphi}E_{\sigma}^{\varphi}(\alpha)\geq \alpha$, as desired.
\qed

\subsection{Proof of Theorem \ref{intermediate-dimension}}
\subsubsection{Dimension of $\Sigma$ in terms of the exponential growth rate of $\mathcal{G}$}
In this subsection, we will characterize $\gamma=\dim_{\rm H}^{\varphi}\Sigma$ in terms of the exponential growth rate of certain family of words.
Recall that $c^{*}=e^{\max_{w\in\Sigma}\varphi(w)}$, and write $c_{*}=e^{\min_{w\in\Sigma}\varphi(w)}$.
For a subcollection of finite words $\mathcal{D}\subset\mathcal{L}$ and each $n\geq 1$, let
\begin{equation*}
    C_\mathcal{D}^n=\left\{ \textbf{i}\in \mathcal{D}: e^{-n}\leq
|[\textbf{i}]|_{\varphi}<c^{*}e^{-n}\right\}.
\end{equation*}
Obviously, $C_\mathcal{L}^n=\Lambda_{e^{-n}}$, as defined in \eqref{Lambda-rho}. For any $\mathcal{D}\subset\mathcal{L}$, we also have the partition:
\begin{equation}\label{CD}
    C_\mathcal{D}^n=\bigsqcup_{0\leq l\leq \lfloor\frac{\log c^{*}}{\log c_{*}}\rfloor, \;l\in\mathbb{N}}\left\{ \textbf{i}\in \mathcal{D}: c_{*}^l e^{-n}\leq
|[\textbf{i}]|_{\varphi}<c_{*}^{l+1} e^{-n}\right\}=:\bigsqcup_{0\leq l\leq\lfloor\frac{\log c^{*}}{\log c_{*}}\rfloor, \;l\in\mathbb{N}}C_\mathcal{D}^n(l).
\end{equation}
Note that if the intersection of two cylinders in $\Sigma$ is nonempty, then one must be contained in the other. By the definition of the metric $d_{\varphi}$, we have $|[\textbf{i}^{-}]|_{\varphi}\geq c_{*}|[\textbf{i}]|_{\varphi}$, thus for any $\mathcal{D}$ and each $l$, the collection of cylinders $\left\{[\textbf{i}]: \textbf{i}\in C_\mathcal{D}^n(l)\right\}$ is disjoint. Recall that under our assumptions, the language of $\Sigma$ admits the decomposition $\mathcal{L}=\mathcal{C}^p\mathcal{G} \mathcal{C}^s$ with $\mathcal{G}$ satisfying the specification property and the pressure gap condition $\sum_{n\geq 1} Z_n(\sigma, \mathcal{C}^p\cup\mathcal{C}^s, -\gamma\varphi)<\infty$. The next lemma indicates that the $\varphi$-Hausdorff dimension of $\Sigma$ can be interpreted as the exponential growth rate of $C_\mathcal{G}^n$.

\begin{lem}\label{lemCG}
Assume that $(\Sigma, \sigma)$ and $\varphi$ satisfy the conditions in Theorem \ref{mainthm}, and $\gamma=\dim_{\rm H}^{\varphi}\Sigma$. Let $C_\mathcal{G}^n(l)$ be defined as in \eqref{CD}, then
\begin{equation}\label{CG}
\limsup_{n\to\infty}\frac{\max_{l}\left(\log\sharp C_\mathcal{G}^n(l)\right)}{n}=\gamma.
\end{equation}
\end{lem}
\begin{proof}
By Theorem \ref{GP} (2), $\dim^\varphi_{\rm H}\Sigma=\dim^\varphi_{\rm B}\Sigma=\gamma$. For any $n\geq 1$ and each $0\leq l\leq \lfloor\frac{\log c^{*}}{\log c_{*}}\rfloor$, since the collection $\left\{[\textbf{i}]: \textbf{i}\in C_\mathcal{G}^n(l)\right\}$ consists of disjoint cylinders of diameter larger than $c_{*}^{l} e^{-n}$, the definition of $\varphi$-Box dimension implies that $N_{c_{*}^{l} e^{-n}}(\Sigma)\geq\sharp C_\mathcal{G}^n(l)$. Thus,
     \begin{equation*}
         \gamma=\dim^\varphi_{\rm B}\Sigma=\lim_{n\to\infty}\frac{\log N_{c_{*}^{l} e^{-n}}(\Sigma)}{n}\geq\limsup_{n\to\infty}\frac{\log\sharp C_\mathcal{G}^n(l)}{n}, \;\text{for any}\; 0\leq l\leq \big\lfloor\frac{\log c^{*}}{\log c_{*}}\big\rfloor.
     \end{equation*}
Therefore,
\begin{equation*} \gamma\geq\limsup_{n\to\infty}\frac{\max_{l}\left(\log\sharp C_\mathcal{G}^n(l)\right)}{n}=:s_0.
\end{equation*}

Now we suppose that $\gamma>s_0$, then one can take $\gamma'\in (s_0, \gamma)$ such that $\max_{l}\sharp C_\mathcal{G}^n(l)<e^{\gamma'n}$ for $n$ large enough.
% One can show that $\max_{l}\sharp C_\mathcal{G}^{\rho}(l)<\rho^{\gamma'}$ for $\rho$ small enough, here $C_\mathcal{G}^{\rho}(l)$ is defined similarly as $C_\mathcal{G}^n(l)$ with $e^{-n}$ replaced by $\rho$.
Let $\mu$ be the measure of maximal dimension in Theorem \ref{extendGP}. For $\mathcal{D}\subset\mathcal{L}$, we will write $\mu(\mathcal{D}):=\mu(\cup_{\textbf{i}\in \mathcal{D}}[\textbf{i}])$ for convenience. Notice that
$\{[\textbf{i}]: \textbf{i}\in C_\mathcal{L}^n\}$ provides a cover of $\Sigma$, thus $\mu(C_\mathcal{L}^n)=1$ for each $n\geq 1$. By the pigeonhole principle, there exist $0\leq l\leq \lfloor\frac{\log c^{*}}{\log c_{*}}\rfloor$ such that $\mu(C_\mathcal{L}^n(l))>\frac{\log c_{*}}{\log c^{*}}$. On the other hand, by the weak Gibbs property \eqref{gibbs} and the Dini condition $\sum_{n\geq 1}\text{Var}_n\varphi<\infty$, we have
\begin{align}\label{CLmeasure}
    \mu(C_\mathcal{L}^n(l))
    & \leq \mu(C_\mathcal{G}^n(l)) + \mu(C_\mathcal{L}^n(l)\setminus C_\mathcal{G}^n(l)) \notag \\
    & \lesssim \sum_{\textbf{i}\in C_\mathcal{G}^n(l)} \sup_{w\in [\textbf{i}]} e^{-\gamma S_{|\textbf{i}|}\varphi(w)} + \sum_{\textbf{i}\in C_\mathcal{L}^n(l)\setminus C_\mathcal{G}^n(l)} \sup_{w\in [\textbf{i}]} e^{-\gamma S_{|\textbf{i}|}\varphi(w)}.
\end{align}

For the first term in \eqref{CLmeasure}, we have
\begin{equation}\label{gcount}
   \sum_{\textbf{i}\in C_\mathcal{G}^n(l)} \sup_{w\in [\textbf{i}]} e^{-\gamma S_{|\textbf{i}|}\varphi(w)}= \sum_{\textbf{i}\in C_\mathcal{G}^n(l)} (|[\textbf{i}]|_{\varphi})^{\gamma}\leq (c_{*}^{l+1} e^{-n})^{\gamma}\sharp C_\mathcal{G}^n(l)
   \lesssim e^{-(\gamma-\gamma')n}.
\end{equation}

For the second term in \eqref{CLmeasure}, note that each $\textbf{i}\in C_\mathcal{L}^n(l)\setminus C_\mathcal{G}^n(l)$ has the decomposition
$\textbf{i}=\textbf{i}^p\textbf{i}^g\textbf{i}^s$ with $\textbf{i}^p\in\mathcal{C}^p, \textbf{i}^g \in\mathcal{G}, \textbf{i}^s\in\mathcal{C}^s$ and $|\textbf{i}^p|+|\textbf{i}^s|\geq 1$. Then
\begin{align}\label{splitineq}
    & \sum_{\textbf{i}\in C_\mathcal{L}^n(l)\setminus C_\mathcal{G}^n(l)} \sup_{w\in [\textbf{i}]} e^{-\gamma S_{|\textbf{i}|}\varphi(w)} \notag\\
     \lesssim &\sum_{k\geq 1} \sum_{\textbf{i}\in C_\mathcal{L}^n(l)\setminus C_\mathcal{G}^n(l)\atop |\textbf{i}^p|+|\textbf{i}^s|=k} \sup_{w\in [\textbf{i}^p]} e^{-\gamma S_{|\textbf{i}^p|}\varphi(w)} \sup_{w\in [\textbf{i}^g]} e^{-\gamma S_{|\textbf{i}^g|}\varphi(w)} \sup_{w\in [\textbf{i}^s]} e^{-\gamma S_{|\textbf{i}^s|}\varphi(w)} \notag \\
     = & \sum_{k\geq \log n} \sum_{w\in\textbf{i}\in C_\mathcal{L}^n(l)\setminus C_\mathcal{G}^n(l)\atop |\textbf{i}^p|+|\textbf{i}^s|=k} \sup_{w\in [\textbf{i}^p]} e^{-\gamma S_{|\textbf{i}^p|}\varphi(w)} \sup_{w\in [\textbf{i}^g]} e^{-\gamma S_{|\textbf{i}^g|}\varphi(w)} \sup_{w\in [\textbf{i}^s]} e^{-\gamma S_{|\textbf{i}^s|}\varphi(w)} \notag \\
     & + \sum_{1\leq k<\log n} \sum_{\textbf{i}\in C_\mathcal{L}^n(l)\setminus C_\mathcal{G}^n(l)\atop |\textbf{i}^p|+|\textbf{i}^s|=k} \sup_{w\in [\textbf{i}^p]} e^{-\gamma S_{|\textbf{i}^p|}\varphi(w)} \sup_{w\in [\textbf{i}^g]} e^{-\gamma S_{|\textbf{i}^g|}\varphi(w)} \sup_{w\in [\textbf{i}^s]} e^{-\gamma S_{|\textbf{i}^s|}\varphi(w)}.
\end{align}

Let $t_n$ be the largest length of words in $C_\mathcal{L}^n(l)$. By the uniform counting result \eqref{Uniformcount2}, the first term in \eqref{splitineq} is bounded from above by
\begin{align}\label{smallpressure}
&\sum_{k\geq \log n} Z_k(\sigma, \mathcal{C}^p\cup\mathcal{C}^s, -\gamma\varphi) Z_{t_n}(\sigma, \mathcal{G}, -\gamma\varphi)\notag\\
\leq & \sum_{k\geq \log n} Z_k(\sigma, \mathcal{C}^p\cup\mathcal{C}^s, -\gamma\varphi) Z_{t_n}(\sigma, \mathcal{L}, -\gamma\varphi) \notag\\
\lesssim &  \sum_{k\geq \log n} Z_k(\sigma, \mathcal{C}^p\cup\mathcal{C}^s, -\gamma\varphi).
\end{align}

For the second term in \eqref{splitineq}, note that for each $k$ and $\textbf{i}\in C_\mathcal{L}^n(l)\setminus C_\mathcal{G}^n(l)$ with $|\textbf{i}^p|+|\textbf{i}^s|=k$, we have
\begin{equation}\label{splitineq1}
    (c_{*})^{\gamma (k+l)} e^{-\gamma n}\leq\sup_{w\in [\textbf{i}^g]} e^{-\gamma S_{|\textbf{i}^g|}\varphi(w)}=(|[\textbf{i}^g]|_{\varphi})^{\gamma}\leq (c_{*})^{\gamma(l+1)}(c^{*})^{\gamma k} e^{-\gamma n}.
\end{equation}
Recall that $\max_{l}\sharp C_\mathcal{G}^{n}(l)<e^{\gamma'n}$ for $n$ large enough, it follows immediately that
\begin{equation}\label{splitineq2}
    \sharp\left\{\textbf{i}\in C_\mathcal{L}^n(l)\setminus C_\mathcal{G}^n(l): |\textbf{i}^p|+|\textbf{i}^s|=k\right\}\lesssim m^k e^{\gamma' n}.
\end{equation}
Plugging \eqref{splitineq1} and \eqref{splitineq2} into the second term in \eqref{splitineq}, we obtain an upper bound
\begin{align}\label{splitineq3}
    &\sum_{1\leq k<\log n} \sum_{\textbf{i}\in C_\mathcal{L}^n(l)\setminus C_\mathcal{G}^n(l)\atop |\textbf{i}^p|+|\textbf{i}^s|=k} \sup_{w\in [\textbf{i}^p]} e^{-\gamma S_{|\textbf{i}^p|}\varphi(w)} \sup_{w\in [\textbf{i}^g]} e^{-\gamma S_{|\textbf{i}^g|}\varphi(w)} \sup_{w\in [\textbf{i}^s]} e^{-\gamma S_{|\textbf{i}^s|}\varphi(w)} \notag\\
    \lesssim & \sum_{1\leq k<\log n} m^k e^{\gamma' n} (c^{*})^{\gamma k} e^{-\gamma n}.
\end{align}

Plugging \eqref{gcount}, \eqref{smallpressure} and \eqref{splitineq3} into \eqref{CLmeasure}, and considering the pressure gap condition, we will see that $\mu(C_\mathcal{L}^n(l))\rightarrow 0$ as $n\rightarrow\infty$, which contradicts with $\mu(C_\mathcal{L}^n(l))>\frac{\log c_{*}}{\log c^{*}}$. Thus, we have
\begin{equation*}
\gamma\leq\limsup_{n\to\infty}\frac{\max_{l}\left(\log\sharp C_\mathcal{G}^n(l)\right)}{n}.
\end{equation*}

%where for the last inequality we have used the definition of $s_B$.  Since $s>s_B$, the last term in the above inequalities tends to 0 as $N\to \infty$.  We thus have $ \dim^\varphi_{\rm H} \Sigma_\mathcal{G}\le s$.
\end{proof}

% Let
% \[s_B=\limsup_{n\to\infty}\frac{\log|C_\mathcal{G}^n|}{n}.\]
% We clearly have $s_B\le \overline{\dim}^\varphi_{\rm B} \Sigma_\mathcal{G}$.

%\begin{lem}\label{claim-1}
%For any $\alpha\in [0, \gamma]$, there exists a subshift $X$ whose $\varphi$-Hausdorff dimension is $\alpha$.
%\end{lem}

\subsubsection{Proof of Theorem \ref{intermediate-dimension}.}
Our goal is show that for any $\alpha\in [0, \gamma]$, there exists a $\sigma$-invariant measure with $\varphi$-Hausdorff dimension $\alpha$. Since Theorem \ref{GP} implies the existence of measure of maximal dimension, we only need to deal with the case $\alpha<\gamma$. Let us fix $\alpha\in [0, \gamma)$.

{\bf Step 1.} Let us fix a small $\epsilon_1\in (0, \frac{\gamma-\alpha}{2})$.  By Lemma \ref{lemCG},  we have
\[
\limsup_{n\to\infty}\frac{\max_{0\leq l\leq \lfloor\frac{\log c^{*}}{\log c_{*}}\rfloor}\left(\log\sharp C_\mathcal{G}^n(l)\right)}{n}=\gamma >\alpha.
\]
Hence there exist arbitrarily large $n_1$ and a collection $A_1\subset C_\mathcal{G}^{n_1}(l)$ with $0\leq l\leq \lfloor\frac{\log c^{*}}{\log c_{*}}\rfloor$ such that
\[
    \frac{\log\sharp A_1}{n_1}\in (\alpha+\frac{\epsilon_1}{2},\alpha+\epsilon_1).
\]
Note that the cylinders in $\{\textbf{u}: \textbf{u}\in A_1\}$ are disjoint.

Let $\mu_1=p^\N$ be the Bernoulli measure on $A_1^\N$, where $p=(p_{\textbf{u}})_{\textbf{u}\in A_1}$ with $p_{\textbf{u}}=\frac{1}{\sharp A_1}$ for all $\textbf{u}\in A_1$.
Let us choose and fix a $\mu_1$-generic point $u=\textbf{u}_1\textbf{u}_2\cdots\in A_1^\N$. By the specification property of $\mathcal{G}$, there exists a sequence $\{\textbf{v}_i\}_{i\ge 1}\subset\mathcal{L}\cap \Lambda^\tau$ such that
\[
\textbf{u}_1\textbf{v}_1\textbf{u}_2\cdots \textbf{u}_{n-1}\textbf{v}_{n-1}\textbf{u}_n\in \mathcal{L} \ \ \textrm{ for all } n\ge 1.
\]
Write $\omega=\textbf{u}_1\textbf{v}_1\textbf{u}_2\textbf{v}_2\cdots $.  Then we have $\omega\in \Sigma$.  Since for each $k\ge 1$, $\textbf{u}_k\textbf{v}_k\in A_1\Lambda^\tau$,  we can also view $\omega$ as an element in $(A_1\Lambda^\tau)^\N$.

For each $n\ge 1$, let
\[
\eta_n'=\frac{1}{n}\sum_{k=1}^{n-1}\delta_{\{\textbf{u}_k\textbf{v}_k\cdots\}}
\]
and
\[
\eta_n=\frac{1}{\tilde{n}}\sum_{k=0}^{\tilde{n}}\delta_{\{\sigma^k\omega\}}
\]
where $\tilde{n}=|\textbf{u}_1\textbf{v}_1 \cdots \textbf{u}_{n}\textbf{v}_{n} |-1$.
For $k\ge 1$,  denote $i_k=|\textbf{u}_k\textbf{v}_k|$.  Note that $\eta_n$ may be rewritten as
\[
\eta_n=\frac{1}{\tilde{n}}\sum_{k=1}^{n-1}\sum_{j=0}^{i_k-1}\delta_{\{\sigma^j(\textbf{u}_k\textbf{v}_k\cdots)\}}.
\]
Denote $m_1=\max\{|\textbf{u}|: \textbf{u}\in A_1\}$.  For $1\le \ell \le m_1+\tau$, let $W_\ell=\{k\in \N: i_k=\ell\}$.  Let
\[
\eta_{n,\ell} = \frac{1}{n}\sum_{\substack{1\le k\le n-1 \\ k\in W_\ell}}\delta_{\{\textbf{u}_k\textbf{v}_k\cdots\}}.
\]
Then we have
\begin{equation}\label{eq:proof:Lemma-claim-1: 1}
\eta_n'=\sum_{\ell=1}^{m_1+\tau}\eta_{n,\ell}
\end{equation}
and
\begin{equation}\label{eq:proof:Lemma-claim-1: 2}
\eta_n=\frac{n}{\tilde{n}}\sum_{\ell=1}^{m_1+\tau}\sum_{j=0}^{\ell-1}\sigma^j(\eta_{n,\ell}).
\end{equation}
Let $(n_k)_k\subset \N$ be a subsequence such that $\frac{n}{\tilde{n}}$ converges to some limit $\rho$ and $\eta_{n,\ell}$ converges to $\eta_{\infty,\ell}$,  $1\le \ell \le m_1+\tau$.  Let
\begin{equation}\label{eq:proof:Lemma-claim-1: 3}
\nu_1'=\sum_{\ell=1}^{m_1+\tau}\eta_{\infty,\ell}
\end{equation}
and
\begin{equation}\label{eq:proof:Lemma-claim-1: 4}
\nu_1=\rho \sum_{\ell=1}^{m_1+\tau}\sum_{j=0}^{\ell-1}\sigma^j(\eta_{\infty,\ell}).
\end{equation}
Thus $\nu_1$ is the limit of $\eta_{n_k}$ as $k\to\infty$.  It follows that  $\nu_1$ is $\sigma$-invariant  and supported on $\Sigma$ (since $\omega\in \Sigma$).

Let us estimate the dimensions of $\nu_1'$ and $\nu_1$.  Observe that the measure $\nu_1'$ is supported on $(A_1\Lambda^\tau)^\N$ and for each $\textbf{u}_1'\textbf{v}_1' \cdots \textbf{u}_{n}'\textbf{v}_{n}'\in (A_1\Lambda^\tau)^n $ (with $\textbf{u}_1',\ldots, \textbf{u}_k'\in A_1$ and $\textbf{v}_1',\ldots, \textbf{v}_k'\in \Lambda^\tau$), we have
\begin{equation}\label{eq:proof:Lemma-claim-1: 5}
\nu_1'([\textbf{u}_1'\textbf{v}_1'\textbf{u}_2'\textbf{v}_2' \cdots \textbf{u}_{n}'\textbf{v}_{n}'])\le \mu_1([\textbf{u}_1'\textbf{u}_2'\cdots \textbf{u}_{n}'])=\left(\sharp A_1\right)^{-n}.
\end{equation}
On the other hand,  writing $\textbf{u}=\textbf{u}_1'\textbf{v}_1' \cdots \textbf{u}_{n}'\textbf{v}_{n}'$,  we have
\[
\diam([\textbf{u}])=|[\textbf{u}]|_\varphi = e^{-\min_{x\in [\textbf{u}]}S_{|\textbf{u}|}\varphi(x) } \ge e^{-n\tau a^*}\prod_{k=1}^n e^{-\max_{x\in [\textbf{u}_k']}S_{|\textbf{u}_k'|}\varphi(x) },
\]
where $a^*=\max_{x\in \Lambda^\N}\varphi(x)$.
Since $\varphi$ satisfies the Dini condition $\sum_{n\geq 1}\text{Var}_n\varphi<\infty$, there exists $b>0$ such that for all $\textbf{v}\in \mathcal{L}$,
\[
\max_{x\in [\textbf{v}]}S_{|\textbf{v}|}\varphi(x) \le \min_{x\in [\textbf{v}]}S_{|\textbf{v}|}\varphi(x) +b.
\]
Thus we have
\begin{equation}\label{eq:proof:Lemma-claim-1: 6}
|[\textbf{u}]|_\varphi \ge  e^{-n\tau a^*}e^{-nb}\prod_{k=1}^n |[\textbf{u}_k']|_\varphi.
\end{equation}
Since $\textbf{u}_k'\in C_{\mathcal{G}}^{n_1}$, by definition, we have
\[
e^{-n_1}\le |[\textbf{u}_k']|_\varphi \le c_*e^{-n_1} ,
\]
where we recall that $c_*=e^{\min_{x\in \Lambda^\N}\varphi(x)}$.
Putting the above estimates together, we obtain that for all $\textbf{u}\in (A_1\Lambda^\tau)^n$,
\begin{equation}\label{eq:proof:Lemma-claim-1: 7}
|[\textbf{u}]|_\varphi \ge  e^{-n(n_1+C)},
\end{equation}
where $C>0$ is some constant only depending on $\varphi$ and $\Lambda^\N$.  By similar arguments, we can show that for each
$\textbf{u}\in (A_1\Lambda^\tau)^n$,
\begin{equation}\label{eq:proof:Lemma-claim-1: 8}
|[\textbf{u}]|_\varphi \le  e^{-n(n_1-C')}
\end{equation}
for some constant $C'>0$ that only depends on $\varphi$ and $\Lambda^\N$.
By \eqref{eq:proof:Lemma-claim-1: 5} and \eqref{eq:proof:Lemma-claim-1: 7},  we have for all $\textbf{u}\in (A_1\Lambda^\tau)^n$,
\[
\frac{\log\nu_1'([\textbf{u}])}{\log |[\textbf{u}]|_\varphi } \ge \frac{\log \sharp(A_1)}{n_1+C}.
\]
As $C$ only depends on $\varphi$ and $\Lambda^\N$, we may assume $n_1$ was chosen large enough so that
\[
 \frac{\log \sharp(A_1)}{n_1+C}\ge  \frac{\log \sharp(A_1)}{n_1}-\frac{\epsilon_1}{4}.
\]
Since $ \frac{\log \sharp(A_1)}{n_1}\ge \alpha+\frac{\epsilon_1}{2}$, we conclude that
\begin{equation}\label{eq:proof:Lemma-claim-1: 9}
\liminf_{n\to\infty}\frac{\log\nu_1'([x|^n_1])}{\log |[x|^n_1]|_\varphi } \ge \alpha+\frac{\epsilon_1}{4}
\end{equation}
for $\nu_1'$-a.e.  $x$.  By the relations \eqref{eq:proof:Lemma-claim-1: 3} and \eqref{eq:proof:Lemma-claim-1: 4} concerning $\nu_1'$ and $\nu_1$, and the fact that the shift map $\sigma$ is bi-Lipschitz outside finitely many points, we infer that for $\nu_1$-a.e.  $x$, the estimate \eqref{eq:proof:Lemma-claim-1: 9} still holds true if we replace $\nu_1'$ by $\nu_1$.  It follows that $\dim_{\rm H}^{\varphi}\nu_1\ge \alpha+\frac{\epsilon_1}{4}$. Moreover,  if $\nu_1=\int \nu_{1,\omega}d\eta(\omega)$ is the ergodic decomposition of $\nu_1$, then for $\eta$-a.e.  $\omega$, we also have $\dim_{\rm H}^{\varphi}\nu_{1,\omega}\ge \alpha+\frac{\epsilon_1}{4}$. That is, almost every ergodic component of $\nu_1$ still has $\varphi$-Hausdorff dimension $\ge  \alpha+\frac{\epsilon_1}{4}$.

Let us record a property of $\nu_1$ that we will use later. By the estimates \eqref{eq:proof:Lemma-claim-1: 7} and \eqref{eq:proof:Lemma-claim-1: 8}, and the fact that $\nu_1'$ is supported on $(A_1\Lambda^\tau)^\N$, we have for $\nu_1'$-a.e.  $x$,
\begin{equation}\label{eq:proof:Lemma-claim-1: 10}
\lambda_1-o_{n_1}(1)\le \liminf_{n\to\infty}\frac{\log|[x|^n_1]|_\varphi }{-n }\le \limsup_{n\to\infty}\frac{\log|[x|^n_1]|_\varphi }{-n }\le \lambda_1+o_{n_1}(1),
\end{equation}
where $\lambda_1=\frac{1}{\sharp A_1}\sum_{\textbf{u}\in A_1}|\textbf{u}|$.  Again, by the relation between $\nu_1'$ and $\nu_1$ (recall\eqref{eq:proof:Lemma-claim-1: 3} and \eqref{eq:proof:Lemma-claim-1: 4})  and the almost bi-Lipschitz property of the shift map $\sigma$, we infer that for $\nu_1$-a.e.  $x$, the estimate \eqref{eq:proof:Lemma-claim-1: 10} still holds.   By the classical ergodic theorem,  if $\mu$ is an ergodic measure on $(\Lambda^\N,\sigma)$, then for $\mu$-a.e.  $x$,
\begin{equation}\label{eq:proof:Lemma-claim-1: 11}
\lim_{n\to\infty}\frac{\log|[x|^n_1]|_\varphi }{-n }=\int \varphi(x)d\mu(x)=:\lambda(\mu).
\end{equation}
Recall that $\nu_1=\int \nu_{1,\omega}d\eta(\omega)$ is the ergodic decomposition of $\nu_1$.  Then from the fact that \eqref{eq:proof:Lemma-claim-1: 10} holds for $\nu_1$-a.e.  $x$,  we infer that  for $\eta$-a.e.  $\omega$, and $\nu_{1,\omega}$-a.e.  $x$,  the estimate \eqref{eq:proof:Lemma-claim-1: 10}  still holds.  Thus by \eqref{eq:proof:Lemma-claim-1: 11},  it follows that for $\eta$-a.e.  $\omega$,
\begin{equation}\label{eq:proof:Lemma-claim-1: 12}
\int \varphi(x)d\nu_{1,\omega}(x)\in (\lambda_1-o_{n_1}(1),\lambda_1+o_{n_1}(1)).
\end{equation}

Now,  let
\[
X_1'=(A_1\Lambda^\tau)^\N  \  \textrm{ and } \  X_1=\bigcup_{j=0}^{m_1+\tau-1}\sigma^{j} X_1'.
\]
Then $X_1$ is compact and shift-invariant.  Since $\sigma$ is bi-Lipschitz outside finitely many points,  we have
\[
\dim_{\rm H}^{\varphi}X_1'=\dim_{\rm H}^{\varphi}X_1 \ \textrm{ and } \  \overline{\dim}_{\rm B}^{\varphi} X_1'=\overline{\dim}_{\rm B}^{\varphi}X_1.
\]
Observe  that $\nu_1'$ is supported on $X_1'$ and $\nu_1$ is supported on $X_1$.

Let us estimate the upper box dimension of $X_1'$. We have
\[
 \overline{\dim}_{\rm B}^{\varphi}X_1' \le \limsup_{n\to\infty} \frac{\log\sharp(A_1\Lambda^\tau)^n}{-\log \max_{\textbf{u}\in (A_1\Lambda^\tau)^n}|[\textbf{u}]|_\varphi }.
\]
Using the fact  $\log\sharp(A_1\Lambda^\tau)^n=n(\log\sharp A_1+\tau\log |\Lambda|)$ and by the estimate  \eqref{eq:proof:Lemma-claim-1: 8},  we get
\[
\overline{\dim}_{\rm B}^{\varphi}X_1' \le \frac{\log\sharp A_1+\tau\log |\Lambda|}{n_1-C'}.
\]
As $C'$ only depends on $\Lambda^\N$ and $\varphi$, we may assume $n_1$ was chosen large enough so that we have
\[
 \frac{\log\sharp A_1+\tau\log |\Lambda|}{n_1-C'} \le  \frac{\log\sharp A_1}{n_1}+\epsilon_1\le \alpha+2\epsilon_1.
\]
Hence we have
\[
\dim_{\rm H}X_1\le \overline{\dim}_{\rm B}^{\varphi}X_1 =\overline{\dim}_{\rm B}^{\varphi}X_1' \le \alpha+2\epsilon_1.
\]

In summary,  we have constructed a set $X_1$,  which is $\sigma$-invariant and compact, and a $\sigma$-invariant measure $\nu_1$, supported both on $X_1$ and on $\Sigma$,  such that
\begin{itemize}
\item[(1)] $ \dim_{\rm H}^{\varphi}X_1, \dim_{\rm H}^{\varphi}\nu_1\in (  \alpha+\frac{\epsilon_1}{4}, \alpha+2\epsilon_1)$;
\item[(2)] $\int \varphi(x)d\nu_1(x)\in (\lambda_1-o_{n_1}(1),\lambda_1+o_{n_1}(1))$.
\end{itemize}
Up to replacing $\nu_1$ by one of its ergodic component, we can also assume that $\nu_1$ is ergodic.

{\bf Step 2.}  Let us now fix $\epsilon_2>0$ with $\epsilon_2<\epsilon/16$. Let $a_2=\min_{\textbf{u}\in A_1\Lambda^\tau}\min_{x\in [\textbf{u}]}S_{|\textbf{u}|}\varphi(x)$. Let us fix a large $n_2$ that we will specify later.  Let
\[
\mathcal{C}^{n_2}_2=\left\{\textbf{u}\in (A_1\Lambda^\tau)^*: e^{-n_2}\le |[\textbf{u}]|_{\varphi}< e^{a_2}e^{-n_2} \right\}.
\]
Then $\{[\textbf{u}]:\textbf{u}\in \mathcal{C}^{n_2}_2 \}$ forms a partition of $X_1'=(A_1\Lambda^\tau)^\N$.
Since $\nu_1'$ is supported both on $X_1'$ and on $\Sigma$, we have $\dim_{\rm H} X_1'\cap\Sigma\ge \dim_{\rm H}\nu_1'\ge \alpha+\frac{\epsilon_1}{4}$.
Thus for large enough $n_1$, we have
\[
\frac{\log \sharp\{ \textbf{u}\in (A_1\Lambda^\tau)^*: \textbf{u}\in \mathcal{C}^{n_2}_2\cap \mathcal{L} \}}{n_2} \ge  \alpha+\frac{\epsilon_1}{8}.
\]
Since $\epsilon_2<\epsilon/16$, assuming $n_1$ is large enough, there exists $A_2'\subset  \mathcal{C}^{n_2}_2\cap \mathcal{L}$ such that
\[
\frac{\log \sharp A_2'}{n_2} \in (\alpha+\frac{\epsilon_2}{2}, \alpha+\epsilon_2).
\]
For each $\textbf{u}'\in A_2'$, let $\textbf{u}$ be the word obtained by erasing the last $\tau$ letters of $\textbf{u}'$. Thus, if $\textbf{u}'\in A_2'$ with $\textbf{u}'=\textbf{u}_1\textbf{v}_1\cdots \textbf{u}_k\textbf{v}_k\in (A_1\Lambda^\tau)^k$, the
$\textbf{u}=\textbf{u}_1\textbf{v}_1\cdots \textbf{u}_k$.  In particular, we have $\textbf{u}\in (A_1\Lambda^\tau)^{k-1}A_1$.
Let
\[
A_2=\{\textbf{u}: \textbf{u}'\in A_2'\}.
\]
Note that the cylinders in $A_2'$ are disjoint, so are those in $A_2$.

Now, we can repeat a similar argument as in  {\bf Step 1.} to construct measures $\nu_2'$ and $\nu_2$ such that
\begin{itemize}
\item[(1)] the measure $\nu_2'$ is supported both on $\Sigma$ and on $X_2'=(A_2\Lambda^\tau)^\N$;
\item[(2)] the measure $\nu_2$ is $\sigma$-invariant and ergodic, it is supported both on $\Sigma$ and on $X_2=\bigcup_{j=0}^{m_2+\tau-1}\sigma^{j} X_2'$, where $m_2=\max\{| \textbf{u}|: u\in A_2\}$;
\item[(3)] $\dim_{\rm H}^{\varphi}X_2'=\dim_{\rm H}^{\varphi}X_2\in (\alpha+\frac{\epsilon_2}{4},\alpha+2\epsilon_2)$;
\item[(4)] $\dim_{\rm H}^{\varphi}\nu_2'=\dim_{\rm H}^{\varphi}\nu_2\in (\alpha+\frac{\epsilon_2}{4},\alpha+2\epsilon_2)$;
\item[(5)] there exists $\lambda_2>0$ such that $\int \varphi d\nu_2\in (\lambda_2-o_{n_2}(1),\lambda_2+o_{n_2}(1))$.
\end{itemize}

{\bf Inductive Steps.}  We can continue the above process so that in the end we will construct three sequences $(X_n)_n,  (\nu_n)_n$ and $(\epsilon_n)_n$ such that $\epsilon_n\searrow 0$ as $n\to \infty$ and for each $n$ we have
 \begin{itemize}
\item[(1)] $X_n$ is $\sigma$-invariant and compact,
\item[(2)] $X_{n+1}\subset X_{n}$,
\item[(3)] $\nu_n$ is $\sigma$-invariant and ergodic, it is supported both on $\Sigma$ and on $X_n$,
\item[(3)] $\dim_{\rm H}^{\varphi}X_n ,  \dim_{\rm H}^{\varphi}\nu_n\in (\alpha+\frac{\epsilon_n}{4},\alpha+2\epsilon_n)$,
\item[(5)] there exists $\lambda_n>0$ such that $\int \varphi d\nu_n\in (\lambda_n-\epsilon_n,\lambda_n+\epsilon_n)$.
\end{itemize}

Let \[X_\infty=\bigcap_{n\ge 1} X_n.\]
Up to taking a subsequence of $\{\nu_n\}_n$, we may assume that
\[
\nu_n \to  \nu_\infty \ \textrm{ as } n\to \infty,
\]
where $\nu_\infty$ is a $\sigma$-invariant measure.
Since for each $n$,  $\nu_n$ is supported both on $\Sigma$ and on $X_n$, the measure $\nu_\infty$ is supported on $\Sigma$ and on $X_\infty$.  Let $F$ denote the topological support of $\nu_\infty$.  Then $F\subset \Sigma$ and it  is $\sigma$-invariant and compact.  We claim that $\dim_{\rm H}^{\varphi} F=\alpha$.
Since $\dim_{\rm H}^{\varphi} X_n\le \alpha+2\epsilon_n$ and $\epsilon_n \searrow 0$,  we have
\[
\dim_{\rm H}^{\varphi} F\le \limsup_{n}(\alpha+2\epsilon_n)=\alpha.
\]
Since $\nu_\infty$ is supported on $F$,  it suffices to show that  $\dim_{\rm H}^{\varphi}\nu_\infty\ge \alpha$.
To show this,  we first notice that since measure entropy is upper semi-continuous on symbolic spaces,  we have
\begin{equation}\label{eq:proof:Lemma-claim-1: 13}
h(\nu_\infty,\sigma)\ge \limsup_{k\to \infty} h(\nu_{n_k},\sigma).
\end{equation}
Let $\nu_\infty=\int \nu_\omega d\eta(\omega)$ be the ergodic decomposition of $\nu_\infty$. It is well known that
\[
h(\nu_\infty,\sigma)=\int h(\nu_\omega,\sigma) d\eta(\omega).
\]
On the other hand, since $\nu_\infty=\lim_{n}\nu_{n}$ and $\varphi$ is continuous, we have
\begin{equation}\label{eq:proof:Lemma-claim-1: 14}
\lambda(\nu_\infty)=\int\varphi(x)d\nu_{\infty}=\lim_{n\to\infty}\int\varphi(x) d\nu_{n}.
\end{equation}
Recall that for each $n$, $\dim_{\rm H}^{\varphi} \nu_{n}\in (\alpha-o_n(1),\alpha+o_n(1))$ and we have (noticing that $\nu_{n}$ is ergodic)
\[
\dim_{\rm H}^{\varphi}\nu_{n}=\frac{h(\nu_{n},\sigma)}{\int \varphi(x)d\nu_{n}(x)}.
\]
Combining this with \eqref{eq:proof:Lemma-claim-1: 13} and \eqref{eq:proof:Lemma-claim-1: 14},  we get
\begin{equation}\label{eq:proof:Lemma-claim-1: 15}
\frac{h(\nu_{\infty},\sigma)}{\int \varphi(x)d\nu_{\infty}(x)}\ge \alpha.
\end{equation}
On the other hand, since for $\eta$-a.e.  $\omega$, $\nu_\omega$ is supported on $F$, we must have
\begin{equation}\label{eq:proof:Lemma-claim-1: 16}
\dim_{\rm H}^{\varphi}\nu_\omega=\frac{h(\nu_{\omega},\sigma)}{\int \varphi(x)d\nu_{\omega}(x)}\le \dim_{\rm H}^{\varphi} F\le \alpha.
\end{equation}
Now in view of the facts $$h(\nu_\infty,\sigma)=\int h(\nu_\omega,\sigma) d\eta(\omega) \  \textrm{ and }\  \int \varphi(x)d\nu_{\infty}(x)=\int\left(\int \varphi(x)d\nu_{\infty}(x)\right)d\eta(\omega),$$
it is easily inferred from \eqref{eq:proof:Lemma-claim-1: 15}  that the following must hold  for $\eta$-a.e. $\omega$:
\[
\dim_{\rm H}^{\varphi}\nu_\omega= \frac{h(\nu_{\omega},\sigma)}{\int \varphi(x)d\nu_{\omega}(x)}\geq\alpha.
\]
This together with \eqref{eq:proof:Lemma-claim-1: 16} implies that $\dim_{\rm H}^{\varphi}\nu_\omega=\alpha$ holds for $\eta$-a.e. $\omega$, in particular,
$\dim_{\rm H}^{\varphi}\nu_\infty=\alpha.$
\qed

%\begin{rem}
%   {\color{blue} How about the packing dimension of the level set? Since the limit measure in the following proof is exact dimensional, it is plausible that the packing dimension of $E_f(\alpha)$ is $\alpha$. }
%\end{rem}

%\begin{defn}
% We define the average (normalized) entropy of $x$ as
%\begin{equation*}
%  \widetilde{h}(x):=\sup_{\nu\in V(x)}\frac{h(\nu, T)}{\chi(\nu)}.
%\end{equation*}
%\end{defn}

\section{Proof of Theorem \ref{beta} and \ref{CER}}
\subsection{Proof of Theorem \ref{beta}}
Let $\varphi(\omega)=\log\beta$ for any $\omega\in\Sigma_\beta$, then $\varphi$ induces the usual metric $d_\varphi(\omega, \omega')=\beta^{-n}$, where $n$ is the minimal integer such that $\omega_{n+1}\neq \omega'_{n+1}$. With this metric $d_\varphi$, we use $\dim_{\rm H}$ to replace $\dim_{\rm H}^\varphi$.  Now we project the symbolic space to Euclidean space by
the natural projection $\pi_\beta: \Sigma_\beta\to [0,1]$ which is defined as $\pi_\beta(\omega)=\sum_{n=1}^\infty\frac{\omega_n}{\beta^n}$ for $\omega=\omega_1\omega_2\dots\in\Sigma_\beta$. We need the following two propositions.
\begin{prop}\label{beta-spec}
  Let $\beta>1$. Then the $\beta$-shift $\Sigma_\beta$ and the potential function $\varphi$ satisfy Condition (1) and (2) in Theorem \ref{mainthm}.
\end{prop}
\begin{proof}
  For Condition (1), recall that a word $u\in\mathcal{L}(\Sigma_\beta)$ (or cylinder $[u]$) is said to be full if $\sigma^{|u|}([u])=\Sigma_\beta$. Choose
  $$\mathcal{G}=\{u\in\mathcal{L}(\Sigma_\beta): \text{the cylinder}\ [u]\ \text{is full}\},$$
  $\mathcal{C}^p=\emptyset$, and $\mathcal{C}^s$ is the set of all prefixes of the infinite $\beta$-expansion of 1 and the empty word. Since every nonfull word $u\in\mathcal{L}(\Sigma_\beta)$ can be written as the concatenation of a full word and a prefix of the infinite $\beta$-expansion of 1 (see \cite[Corollary 3.4]{BW14} or \cite[Corollary 3.9]{LL18}), we know that $\mathcal{L}(\Sigma_\beta)\subset \mathcal{C}^p\mathcal{G}\mathcal{C}^s$, while for full word, we just choose empty word in $\mathcal{C}^s$. Moreover, $\mathcal{G}$ satisfies the specification property with $\tau=0$, which is given by the fact that the concatenation of any two full words is still full (see \cite[Lemma 3.2]{BW14} or \cite[Proposition 3.2 (1)]{LL18}).

  For Conditioin (2), note that $\beta^n\leq\sharp\Sigma_{\beta}^n\leq (\beta-1)^{-1}\beta^{n+1}$ (see \cite{Ren57}), thus Theorem \ref{GP} implies $\gamma=\dim_{\rm H}\Sigma_\beta=1$. Since $\#(\mathcal{C}^s\cap \mathcal{A}_\beta^n)=1$ for all $n\geq 1$, then checking $\sum_{n\geq 1} Z_n(\sigma, \mathcal{C}^p\cup\mathcal{C}^s, -\gamma\varphi)<\infty$ is amount to checking $\sum_{n\geq 1}e^{-S_n\varphi(\omega^1)}<\infty$, which is obvious, here $\omega^1$ is the infinite $\beta$-expansion of 1.
\end{proof}

The second propostition sataes that the set $E\subset\Sigma_\beta$ has same dimensions with its image $\pi_\beta(E)$ respectively.
\begin{prop}\label{projectionsets}
	Let $\beta>1$ and $E\subset\Sigma_\beta$. Then
	$$\dim_{\rm H} E=\dim_{\rm H} \pi_\beta(E),\ \overline{\dim}_{\rm B} E=\overline{\dim}_{\rm B} \pi_\beta(E),\ \underline{\dim}_{\rm B} E=\underline{\dim}_{\rm B} \pi_\beta(E),\ \text{and}\ \dim_{\rm P}E=\dim_{\rm P}\pi_\beta(E).$$
\end{prop}
\begin{proof}
  Note that $\dim_{\rm H} E=\dim_{\rm H} \pi_\beta(E)$ follows from Theorem 1.1 in \cite{Li19}.
  Since $\pi_\beta$ is Lipschitz, we have $\overline{\dim}_{\rm B} E\geq\overline{\dim}_{\rm B} \pi_\beta(E),\ \underline{\dim}_{\rm B} E\geq\underline{\dim}_{\rm B} \pi_\beta(E)$ and $\dim_{\rm P}E\geq\dim_{\rm P}\pi_\beta(E)$. So we just need to check the inverse inequalities. Let $\rho=\beta^{-\ell}$ and $N_\rho(\pi_\beta(E))$ be the smallest number of intervals of length $2\rho$ that cover $\pi_\beta(E)$. By the covering property of $\beta$-expansion (Proposition 4.1 in \cite{BW14}), each such interval can be covered at most $4(\ell+1)$ cylinders of order $\ell$, which implies that these $4(\ell+1)N_\rho(\pi_\beta(E))$ cylinders of order $\ell$ in $\Sigma_\beta$ cover the set $E$. Note that the diameter of every  cylinder of order $\ell$ in $\Sigma_\beta$ under the metric $d_\varphi$ with $\varphi(\omega)=\log\beta$, is $\beta^{-\ell}$, so
  $$N_{\beta^{-\ell}}(E)\leq 4(\ell+1)N_\rho(\pi_\beta(E)).$$
  It follows that $\overline{\dim}_{\rm B} E\leq\overline{\dim}_{\rm B} \pi_\beta(E)$ and $\underline{\dim}_{\rm B}E\leq\underline{\dim}_{\rm B} \pi_\beta(E)$. Furthermore, recall that
  $$\dim_{\rm P}E=\overline{\dim}_{\rm MB}E=\inf\{\sup\overline{\dim}_{\rm B}E_i:\ E\subset\cup_{i\geq 1}E_i\},$$
  we consider any sequence $F_i$ such that $\pi_\beta(E)\subset \cup_{i\geq 1}F_i$, which gives that $E\subset \cup_{i\geq 1}\pi_\beta^{-1}(F_i)$. Thus
  $$\dim_{\rm P}E\leq\sup_{i\geq 1}\overline{\dim}_{\rm B}\pi_\beta^{-1}(F_i)=\sup_{i\geq 1}\overline{\dim}_{\rm B}F_i.$$
  It follows that $\dim_{\rm P}E\leq\dim_{\rm P}\pi_\beta(E)$ by the arbitrariness of $\{F_i\}_{i\geq 1}$.
\end{proof}

\begin{prop}\label{invariant}
Let $\beta>1$ and $A\subset [0, 1)$ be a $T_{\beta}$-invariant closed set
(i.e. $T_{\beta}A\subset A$). Then
$\dim_{\rm H} A=\dim_{\rm B} A.$ In particular,
$\dim_{\rm H} \overline{\mathcal{O}_{T_\beta}(x)}=\dim_{\rm B} \overline{\mathcal{O}_{T_\beta}(x)}.$
\end{prop}
\begin{proof}
Since $\pi_\beta$ is continuous and $A$ is closed, $\pi_\beta^{-1}(A)$ is closed under the usual metric $d_\varphi$ with $\varphi(\omega)=\log\beta$ for $\omega\in\mathcal{A}_\beta^\mathbb{N}$. Note that $\pi_\beta\circ\sigma=T_\beta\circ\pi_\beta$ and $A$ is $T_\beta$-invariant, we know that $\pi_\beta^{-1}(A)$ is $\sigma$-invariant.
So $\pi_\beta^{-1}(A)$ is a subshift). By Theorem 2.1 (ii) in \cite{GP97}, $\dim_{\rm H} \pi_\beta^{-1}(A)=\dim_{\rm B}\pi_\beta^{-1}(A).$ Applying Proposition \ref{projectionsets} gives that
	$$\dim_{\rm H} A=\dim_{\rm H} \pi_\beta^{-1}(A)=\dim_{\rm B}\pi_\beta^{-1}(A)=\dim_{\rm B} A$$
by noting that $\pi_\beta(\pi_\beta^{-1}A)=A$ due to the surjection of $\pi_\beta$.
\end{proof}

Now we retun to the proof of Theorem \ref{beta}, by projecting the level set $E_\sigma(\alpha)$ from symbolic space to Euclidean space, we obtain $$\dim_{\rm H}E_\sigma(\alpha)=\dim_{\rm H}\pi_\beta(E_\sigma(\alpha)).$$
Note that $\pi_\beta(\overline{\mathcal{O}_\sigma(\omega)})=\overline{\mathcal{O}_{T_\beta}(\pi_\beta(\omega))}$ for all $\omega\in\Sigma_\beta$, we have
$$\pi_\beta(E_\sigma(\alpha))=\{x\in [0, 1]: \dim_{\rm H} \overline{\mathcal{O}_{T_\beta}(x)}=\alpha\}=E_{T_\beta}(\alpha).$$
Applying Theorem \ref{mainthm} to $(\Sigma_\beta, \sigma)$, we know that $\dim_{\rm H}E_{T_\beta}(\alpha)=\alpha$ for any $\alpha\in [0, 1]$.
\qed

\subsection{Proof of Theorem \ref{CER}}
It is known that the expanding property guarantees that $(X, T)$ has Markov partitions of arbitrarily small size (see \cite{GP97}, Page 161). In particular, $(X, T)$ is a factor of a subshift of finite type. By the definition of CER, $T: X\rightarrow X$ is an open mapping. Replacing the map $T$ by its power if necessary, we may assume that
\begin{equation*}
  1<a\leq \|(T^n)'(x)\|\leq b<\infty, \forall x\in X
\end{equation*}
for some constants $a, b$. Since $X$ is compact, there exists $r>0$ and $\kappa>1$ such that for any $x, y\in X$ with $d(x, y)\leq r$, we have
\begin{equation}\label{expanding}
  d(T(x), T(y))\geq\kappa d(x, y).
\end{equation}
Let $\{Q_1,\ldots, Q_m\}$ be a Markov partition for $(X, T)$ with $\diam(Q_i)\leq\min\{r, \kappa^{-1}\}$, that is, each $Q_i$ is a closed subset of $X$ and $X=\cup_{i=1}^m Q_i$; the relative interior of $Q_i$ in $X$ is dense in $Q_i$ and the relative interiors of $Q_i$ and $Q_j$ are disjoint for $i\neq j$; each $T(Q_i)$ is a union of $Q_j$. Let $\Lambda=\{1, \ldots, m\}$ and $\Omega=\Lambda^{\mathbb{N}}$, now we define the symbolic extension of $(X, T)$ by
\begin{equation*}
  \Sigma_T=\{w\in\Omega:\; Q_{w_i}\subset T(Q_{w_{i-1}})\;\text{for all}\; i\geq 1\}.
\end{equation*}
We call a finite sequence $\tau\in\Lambda^n$ \emph{admissible} if $Q_{\tau_i}\subset T(Q_{\tau_{i-1}})$ for all $1\leq i< n$ and an infinite sequence $w\in\Omega$ is called \emph{admissible} if $w|_1^n$ is admissible for any $n\geq 1$. Then $\Sigma_T$ consists of all admissible sequences in $\Omega$. For any $w\in\Sigma_T$ and $n\geq 1$, set
\begin{equation*}
   Q_n(w):=\bigcap_{j=0}^{n-1}T^{-j}(Q_{w_{i_j}}).
\end{equation*}
By \eqref{expanding}, we have
\begin{equation*}
   \diam{Q_n(w)}\leq\kappa^{-1} \diam{Q_{n-1}(w)}\leq \kappa^{-n}.
\end{equation*}
Consequently, for each $w\in\Sigma_T$, the set $\cap_{n\geq 0}Q_n(w)$ is a singleton, denoted by $\pi(w)$. Thus, we define a surjective map $\pi: \Sigma_T\rightarrow X$, which satisfies $\pi\circ \sigma(w)=T\circ\pi(w)$ for any $w\in\Sigma_T$. Let $\varphi(w)=-\log |T'(\pi(w))|$, which is a strictly positive continuous function on $\Sigma_T$ and induces a metric $d_{\varphi}$ on $\Sigma_T$ as in Section \ref{setting}. The following proposition concerning the regularity property of $\pi: \Sigma_T\rightarrow X$ was proved in \cite{GP97}.

\begin{prop}\label{CER_regularity}(see \cite{GP97}, Page 162-163)
  The projection map $\pi$ from the metric space $(\Sigma_T, d_{\varphi})$ to $(X, d)$ satisfies the hypothesis of Lemma \ref{keylemma1} and Lemma \ref{keylemma2} with $a, b$ arbitrarily close to 1. Consequently, for any $E\subset\Sigma_T$,
  \begin{equation*}
    \dim_{\rm H}E=\dim_{\rm H}\pi(E),\; \overline{\dim}_{\rm B}E=\overline{\dim}_{\rm B}\pi(E)\; \text{and}\;\underline{\dim}_{\rm B}E=\underline{\dim}_{\rm B}\pi(E).
  \end{equation*}
\end{prop}

Now we return to the proof of Theorem \ref{CER}. By Proposition \ref{CER_regularity}, $\dim_{\rm H}E_T(\alpha)=\dim_{\rm H}\pi(E_T(\alpha)).$ Note that $\pi(\overline{\mathcal{O}_\sigma(\omega)})=\overline{\mathcal{O}_{T}(\pi(\omega))}$ for any $\omega\in\Sigma_T$, we have
$$\pi(E_\sigma(\alpha))=\{x\in X: \dim_{\rm H} \overline{\mathcal{O}_{T}(x)}=\alpha\}=E_T(\alpha),$$
Applying Theorem \ref{mainthm} to $(\Sigma_T, \sigma)$, we know that $\dim_{\rm H}E_{T}(\alpha)=\alpha$ for any $\alpha\in [0, \dim_{\rm H}X]$.
\qed

\section{Furstenberg level set}
%\subsection{Extend Shmerkin-Wu Theorem to $\beta$-transformations}
%\begin{thm}\label{Wuextend}
%  Let $\beta_1,\beta_2>1$ with $\beta_1\nsim \beta_2$. Let $A, B\subset [0, 1)$ be closed set and $T_{\beta_1}$-invariant, $T_{\beta_2}$-invariant respectively. Then for all real numbers $u$ and $v$, we have
%  $$\overline{\dim}_{\rm B}(uA+v)\cap B\leq \max\{0, \dim_{\rm H}A+\dim_{\rm H}B-1\}.$$
%\end{thm}
%\begin{proof}
% {\color{red} Need a proof.}
%\end{proof}

%\begin{rem}\label{Wuextend_rem}
%  For any integers $m, n\geq 1$, the above theorem is true for any $T_{\beta_1}^m$-invariant set $A$ and $T_{\beta_1}^n$-invariant set $B$. Indeed, let $\widetilde{A}=\cup_{j=0}^{m-1}T_{\beta_1}^j(A)$. Since $A$ is $T_{\beta_1}^m$-invariant and closed, then $\widetilde{A}$ is $T_{\beta_1}$-invariant and closed.
%  Moreover, since $T_{\beta_1}^j$ is locally bi-Lipschitz, we have $\dim_{\rm H}A=\dim_{\rm H}\widetilde{A}$.
%  Let $\widetilde{B}=\cup_{j=0}^{n-1}T_{\beta_2}^j(B)$, then the same things also hold for $B$ and $\widetilde{B}$. Applying the above theorem to $\widetilde{A}, \widetilde{B}$ gives that
%   $$\overline{\dim}_{\rm B}(uA+v)\cap B\leq\overline{\dim}_{\rm B}(u\widetilde{A}+v)\cap \widetilde{B}\leq \max\{0, \dim_{\rm H}\widetilde{A}+\dim_{\rm H}\widetilde{B}-1\}=\max\{0, \dim_{\rm H}A+\dim_{\rm H}B-1\}.$$
%\end{rem}

\subsection{Proof of Theorem \ref{Furstenbergset}}
The multiplicatively dependent case relies on the next proposition.
\begin{prop}\label{dependentcase}
Let $p, q\geq 2$ be integers with $p\sim q$. Then for any $x\in [0, 1)$, $$\dim_{\rm H}\overline{\mathcal{O}_{T_p}(x)}=\dim_{\rm H}\overline{\mathcal{O}_{T_q}(x)}.$$
\end{prop}
\begin{proof}
Since $p\sim q$, there exist $a, b\in\mathbb{N}$ such that $p^a=q^b$, we write this common value as $m$. Note that
$$\mathcal{O}_{T_p}(x)=\bigcup_{i=0}^{a-1}\mathcal{O}_{T_m}(T_p^ix),$$
we have
$$\overline{\mathcal{O}_{T_p}(x)}=\bigcup_{i=0}^{a-1}\overline{\mathcal{O}_{T_m}(T_p^ix)}$$
since it is a finite union. So
\begin{equation}\label{supdimension}
\dim_{\rm H}\overline{\mathcal{O}_{T_p}(x)}=\max_{0\leq i\leq a-1}\dim_{\rm H}\overline{\mathcal{O}_{T_m}(T_p^ix)}.
\end{equation}
 For any $0\leq i, j\leq a-1$, there exists an affine map $f$ such that $f(T_p^ix)=T_p^jx$, which gives that $f(\mathcal{O}_{T_m}(T_p^ix))=\mathcal{O}_{T_m}(T_p^jx)$. It follows that
 $$\dim_{\rm B}\overline{\mathcal{O}_{T_m}(T_p^ix)}=\dim_{\rm B}\mathcal{O}_{T_m}(T_p^ix)=\dim_{\rm B}\mathcal{O}_{T_m}(T_p^jx)=\dim_{\rm B}\overline{\mathcal{O}_{T_m}(T_p^jx)}.$$
 Since $\overline{\mathcal{O}_{T_m}(T_p^ix)}$ is $T_m$-invariant and closed, by Proposition \ref{invariant} we obtain that
 \begin{equation}\label{ijdimension}
 \dim_{\rm H}\overline{\mathcal{O}_{T_m}(T_p^ix)}=\dim_{\rm B}\overline{\mathcal{O}_{T_m}(T_p^ix)}=\dim_{\rm B}\overline{\mathcal{O}_{T_m}(T_p^jx)}=\dim_{\rm H}\overline{\mathcal{O}_{T_m}(T_p^jx)}.
 \end{equation}
 Together with \eqref{supdimension} and \eqref{ijdimension}, we have $\dim_{\rm H}\overline{\mathcal{O}_{T_p}(x)}=\dim_{\rm H}\overline{\mathcal{O}_{T_m}(x)}$. Also we have $\dim_{\rm H}\overline{\mathcal{O}_{T_q}(x)}=\dim_{\rm H}\overline{\mathcal{O}_{T_m}(x)}$.
 So it completes the proof.
\end{proof}

For the multiplicatively independent case, we need the following proposition.
\begin{prop}\label{upperFurstenberglevel}
Let $p, q\geq 2$ be integers with $p\nsim q$ and $s\in [0, 2]$. Then
$$\dim_{\rm H} F_{p, q}^s\leq \max\{0, s-1\}.$$
\end{prop}
\begin{proof}
    The idea is to modify Wu's arguments for the proof of Theorem 9.4 in \cite{Wu19} by replacing 1 there using $s$. Write
	$$F_1(s)=\{(A, B): A\;\text{is a}\;T_p\;\text{-invariant closed set}, B\;\text{is a}\;T_q\;\text{-invariant closed set}, \dim_{\rm{H}} A+\dim_{\rm{H}} B\leq s\;\},$$
    and
	$$F_2(s)=\{(\widetilde{A}, \widetilde{B}): \widetilde{A}\;\text{is a}\;p^k\;\text{-Cantor set}, \widetilde{B}\;\text{is a}\;q^l\;\text{-Cantor set},
	\dim_{\rm{H}} \widetilde{A}+\dim_{\rm{H}} \widetilde{B}\leq s\;\text{for some}\; k, l\geq 1\},$$
   here by $m$-Cantor we mean a attractor the iterated function system $\{x/m+i/m\}_{i\in\Lambda}$ for some
   $\Lambda\subset\{0, \ldots, m-1\}$. Then for any
   $\epsilon>0$ we have,
\begin{align}\label{include}
   F_{p, q}^s\subset&\{x\in [0, 1): \dim_{\rm H}\overline{\mathcal{O}_{T_p}(x)}+\dim_{\rm H}\overline{\mathcal{O}_{T_q}(x)}\leq s\}\notag\\
   \subset&\bigcup_{(A, B)\in F_1(s)}A\cap B\subset\bigcup_{(\widetilde{A}, \widetilde{B})\in F_2(s+\epsilon)}\widetilde{A}\cap \widetilde{B},
\end{align}
   where the last inclusion is due to \cite[Proposition 9.3]{Wu19}.
   By \cite[Theorem 1.3]{Wu19}, for each $(\widetilde{A}, \widetilde{B})\in F_2(s+\epsilon)$, we have
   $$\dim_{\rm{H}} \widetilde{A}\cap\widetilde{B}\leq \max\{0, \dim_{\rm{H}} \widetilde{A}+\dim_{\rm{H}}\widetilde{B}-1\}\leq\max\{0, s+\epsilon-1\}.$$
   By \eqref{include}, $F_{p, q}^s$ is contained in a countable union of Hausdorff dimension less than $\max\{0, s+\epsilon-1\}$, thus $\dim_{\rm{H}}F_{p, q}^s\leq\max\{0, s+\epsilon-1\}$. Since $\epsilon>0$ is arbitrary, the result follows immediately.
\end{proof}

Now we are ready to give a proof of Theorem \ref{Furstenbergset}.
\begin{proof}[Proof of Theorem \ref{Furstenbergset}]
  (i) $p\sim q$. By Proposition \ref{dependentcase}, $$F_{p,q}^s=\{x\in [0,1): \dim_{\rm H}\overline{\mathcal{O}_{T_p}(x)}=\frac{s}{2}\},$$ which implies that $\dim_{\rm H} F_{p,q}^s=\frac{s}{2}$ by Theorem \ref{beta}.

  (ii) $p\nsim q$. The upper bound follows directly from Proposition \ref{upperFurstenberglevel}.
  The left is to estimate the lower bound. We just need to consider the case $s\in [1,2]$ since the upper bound gives
  $\dim_{\rm H}F_{p,q}^s=0$ whenever $s\in [0,1)$.
  First we apply the intermediate dimension property to $([0, 1), T_p)$, there exists $A\subset [0, 1)$ which is $T_p$-invariant and closed, with $\dim_{\rm H}A=s-1$.
   %For the lower bound let $A\subset [0, 1)$ be a $T_p$-invariant closed set with $\dim_{\rm H}A=s-1$, this is possible by applying the intermediate entropy property and the formula $\dim_{\rm H}A=\frac{h_{top}(A)}{\log p}$.
    % Let $\mu$ be the measure of maximal entropy on $A$, then $\dim_{\rm H}\mu=s-1$.
     By Theorem 1(i) in \cite{GP97}, there exists an $T_p$-invariant ergodic measure $\mu$ on $A$ such that $\dim_{\rm H}\mu=\dim_{\rm H}A=s-1$.
     By Birkhoff's ergodic theorem, $\mu$ almost all $x$ have dense orbits in $A$, i.e., $\overline{\mathcal{O}_{T_p}(x)}=A$ and consequently $\dim_{\rm H}\overline{\mathcal{O}_{T_p}(x)}=\dim_{\rm H}A=s-1$. On the other hand, applying \cite[Theorem 1]{Host95} gives that
      $\mathcal{O}_{T_q}(x)=\{q^n x\ \text{mod}\ 1\}_{n\geq 1}$ is equidisbributed for $\mu$-almost every $x\in [0, 1)$. Consequently $\overline{\mathcal{O}_{T_q}(x)}=[0, 1]$ and $\dim_{\rm H}\overline{\mathcal{O}_{T_q}(x)}=1$ for $\mu$-almost all $x\in [0, 1)$. Therefore, the $\mu$ measure of
  \begin{equation*}
    E_{T_p}(s-1)\cap E_{T_q}(1)=\{x\in [0,1): \dim_{\rm H}\overline{\mathcal{O}_{T_p}(x)}=s-1, \dim_{\rm H}\overline{\mathcal{O}_{T_q}(x)}=1\}
  \end{equation*}
  is full, which implies that $\dim_{\rm H}E_{T_p}(s-1)\cap E_{T_q}(1)\geq \dim_{\rm H}\mu=s-1$. Since $E_{T_p}(s-1)\cap E_{T_q}(1)\subset F_{p, q}^{s}$, we have \begin{equation*}
    \dim_{\rm H}F_{p, q}^{s}\geq \dim_{\rm H}E_{T_p}(s-1)\cap E_{T_q}(1)\geq s-1.
  \end{equation*}
\end{proof}

\subsection{Intersection}
Let $p, q\geq 2$ be two integers and $\alpha_1, \alpha_2\in [0, 1]$. We consider the intersection $E_{T_p}(\alpha_1)\cap E_{T_q}(\alpha_2)$ and want to know its Hausdorff dimension. Proposition \ref{dependentcase} gives the following.
\begin{prop}
Let $p, q\geq 2$ be integers with $p\sim q$, and $\alpha_1, \alpha_2\in [0, 1]$. Then
$$\dim_{\rm H}E_{T_p}(\alpha_1)\cap E_{T_q}(\alpha_2)=\left\{
  \begin{array}{ll}
   \alpha_1, &\text{if}\; \alpha_1=\alpha_2;\\
  0, &\text{if}\; \alpha_1\neq\alpha_2.
  \end{array}
\right.$$
\end{prop}

\begin{prop}
	Let $p, q\geq 2$ be integers with $p\nsim q$, and $\alpha_1, \alpha_2\in [0, 1]$.
$$\dim_{\rm H}E_{T_p}(\alpha_1)\cap E_{T_q}(\alpha_2)\leq\max\{\alpha_1+\alpha_2-1, 0\}.$$
\end{prop}
\begin{proof}
	 Note that
	$$E_{T_p}(\alpha_1)\cap E_{T_q}(\alpha_2)\subset F_{p,q}^{\alpha_1+\alpha_2},$$
	the result follows from Proposition \ref{upperFurstenberglevel}.
\end{proof}

We conjecture that
\begin{con}
  Let $p, q\geq 2$ be integers with $p\nsim q$, and $\alpha_1, \alpha_2\in [0, 1]$. Then
  $$\dim_{\rm H}E_{T_p}(\alpha_1)\cap E_{T_q}(\alpha_2)=\max\{\alpha_1+\alpha_2-1, 0\}.$$
 \end{con}

We remark that $E_{T_p}(\alpha_1)$ is $T_p$-invariant but not closed, in fact, it is even dense in $[0, 1]$. Similarly, $E_{T_q}(\alpha_2)$ is a $T_q$-invariant dense subset in $[0 ,1]$. Thus, \cite[Theorem 1.3]{Wu19} cannot be applied to solve the above conjecture.

\section*{Acknowledgements} We thank Aihua Fan and Shilei Fan for helpful discussions, and thank De-Jun Feng for sharing \cite{Feng} with us. The first author was supported by by NSFC 11901204 and 12271418. The second author was supported NSFC12271176 and Guangdong Natural Science Foundation 2024A1515010946. The first and the third author were supported by the Academy of Finland, project grant No. 318217.
%%%%%%%%%%%%%%%%%%%%%%%%%%%%%%%%%%%%%%%%%%%%%%%%%%%%%%


\begin{thebibliography}{99}
%\bibitem{ABS22} A. Algom, S. Baker, P. Shmerkin,
%\textit{On normal numbers and self-similar measures},
%Adv. Math. 399 (2022), Paper No. 108276, 17 pp.

%\bibitem{Aus20} T. Austin,
%\textit{A new dynamical proof of the Shmerkin-Wu theorem},
%preprint, arXiv: 2009.01292.

%\bibitem{Ber77} A. Bertrand-Mathis,
%\textit{D\'{e}veloppements en base de Pisot et r\'{e}partition modulo 1},
%C. R. Acad. Sci. Paris S\'{e}r. A-B 285(6), A419-A421 (1977).

%\bibitem{Ber79} A. Bertrand-Mathis,
%\textit{R\'{e}partition modulo 1 et d\'{e}veloppement en base $\theta$},
%C. R. Acad. Sci. Paris S\'{e}r. A-B 289(1), A1-A4 (1979).

\bibitem{Bow73} R. Bowen,
\textit{Topological entropy for noncompact sets},
Trans. Amer. Math. Soc. 184 (1973), 125-136.

\bibitem{Bow74} R. Bowen,
\textit{Some systems with unique equilibrium states},
Math. Systems Theory 8 (1974/75), no. 3, 193-202.

\bibitem{Bru98} H. Bruin,
\textit{Dimensions of recurrence times and minimal subshifts},
Dynamical systems (Luminy-Marseille, 1998), 117-124, World Sci. Publ., River Edge, NJ, 2000.

\bibitem{BW14} Y. Bugeaud, B.-W. Wang,
\textit{Distribution of full cylinders and the Diophantine properties of the orbits in beta-expansions},
J. Fractal Geom. 1 (2014), no. 2, 221-241.

\bibitem{Cli22} V. Climenhaga,
\textit{An improved non-uniform specification result}, blog post,
\url{https://vaughnclimenhaga.wordpress.com/2022/06/06/an-improved-non-uniform-specification-result/}.

%\bibitem{CP19} V. Climenhaga, R. Pavlov,
%\textit{One-sided almost specification and intrinsic ergodicity},
%Ergodic Theory Dynam. Systems 39 (2019), no. 9, 2456-2480.

\bibitem{CT12} V. Climenhaga, D. J. Thompson,
\textit{Intrinsic ergodicity beyond specification: beta-shifts, S-gap shifts, and their factors},
Israel J. Math. 192 (2012), no. 2, 785-817.

\bibitem{CT13} V. Climenhaga, D. J. Thompson,
\textit{Equilibrium states beyond specification and the Bowen property},
J. Lond. Math. Soc. (2) 87 (2013), no. 2, 401-427.

\bibitem{CT21} V. Climenhaga, D. J. Thompson,
\textit{Beyond Bowen's specification property},
Thermodynamic formalism, 3-82, Lecture Notes in Math., 2290, Springer (2021).

%\bibitem{Dow11} T. Downarowicz,
%\textit{Entropy in dynamical systems.} New Mathematical Monographs, 18. Cambridge University Press, Cambridge, 2011.

%\bibitem{EW} M. Einsiedler, T. Ward,
%\textit{Entropy in ergodic theory and homogeneous dynamics},
%book draft, available online at %\url{https://tbward0.wixsite.com/books/entropy}.

\bibitem{Fal90} K. Falconer,
\textit{Fractal geometry: mathematical foundations and applications},
Wiley, New York (1990).

%\bibitem{Fal97} K. Falconer,
%\textit{Techniques in fractal geometry},
%John Wiley \& Sons Ltd, 1997.

\bibitem{FLR02} A.-H. Fan, K.-S. Lau and H. Rao \emph{Relationships between different dimensions of a measure}, Monatsh. Math. 135 (2002), no. 3, 191-201.
\bibitem{Feng} D.-J. Feng, A note on the $b$-adic entropies of real numbers, unpublished.
\bibitem{Fur67} H. Furstenberg,
\textit{Disjointness in ergodic theory, minimal sets, and a problem in diophantine approximation},
Math. systems theory 1(1967), 1-49.

\bibitem{Fur69} H. Furstenberg,
\textit{Intersections of Cantor sets and transversality of semigroups}, in Problems in Analysis (Sympos. Salomon Bochner, Princeton Univ., Princeton, N.J., 1969), Princeton Univ. Press, Princeton, N.J., 1970, pp. 41-59.

\bibitem{GP97} D. Gatzouras, Y. Peres,
\textit{Invariant measures of full dimension for some expanding maps},
Ergodic Theory Dynam. Systems 17 (1997), no. 1, 147-167.

%\bibitem{Good41} I. J. Good,
%\textit{The fractional dimensional theory of continued fractions},
%Proc. Camb. Philos. Soc. 37 (1941), 199-228.

\bibitem{Gri72} C. Grillenberger,
\textit{Constructions of strictly ergodic systems. I. Given entropy},
Z. Wahrscheinlichkeitstheorie und Verw. Gebiete 25 (1972/73), 323-334.

%\bibitem{HS2015} M. Hochman, P. Shmerkin,
%\textit{Equidistribution from fractal measures},
%Invent. Math. 202 (2015), no. 1, 427-479.

\bibitem{Host95} B. Host,
\textit{Nombres normaux, entropie, translations},
Israel J. Math. 91 (1995), no. 1-3, 419-428.

\bibitem{HYZ10} W. Huang, X.-D. Ye, G.-H. Zhang,
\textit{Lowering topological entropy over subsets},
Ergodic Theory Dynam. Systems 30 (2010), no. 1, 181-209.

\bibitem{HYZ14} W. Huang, X.-D. Ye, G.-H. Zhang,
\textit{Lowering topological entropy over subsets revisited},
Trans. Amer. Math. Soc. 366 (2014), no. 8, 4423-4442.

\bibitem{LW08} B. Li, J. Wu,
\textit{Beta-expansion and continued fraction expansion},
J. Math. Anal. Appl. 339 (2008), no. 2, 1322-1331.

\bibitem{Li19} Y.-Q. Li,
\textit{Hausdorff dimension of frequency sets in beta-expansions},
Math. Z. 302 (2022), no. 4, 2059-2076.

\bibitem{LL18} Y.-Q. Li, B. Li,
\textit{Distributions of full and non-full words in beta-expansions},
J. Number Theory 190 (2018), 311-332.

\bibitem{LM95} D. Lind, B. Marcus,
\textit{An introduction to symbolic dynamics and coding},
Cambridge University Press, Cambridge, 1995.

\bibitem{Lin95} E. Lindenstrauss,
\textit{Lowering topological entropy},0
J. Anal. Math. 67 (1995), 231-267.

%\bibitem{Lot02} M. Lothaire,
%\textit{Algebraic combinatorics on words, Encyclopedia of Mathematics and its Applications 90}, Cambridge University Press, Cambridge, 2002.

\bibitem{Mat95} P. Mattila,
\textit{Geometry of sets and measures in Euclidean space. Fractals and rectifiability},
Cambridge Studies in Advanced Mathematics, 44. Cambridge University Press, Cambridge, 1995.

%\bibitem{MM12} C. Mauduit, C. G. Moreira,
%\textit{Generalized Hausdorff dimensions of sets of real numbers with zero entropy expansion},
%Ergodic Theory Dynam. Systems 32 (2012), no. 3, 1073-1089.

\bibitem{PYY22} M. J. Pacifico, F. Yang, J.-G Yang,
\textit{Existence and uniqueness of equilibrium states for systems with specification at a fixed scale: an improved Climenhaga-Thompson criterion}, Nonlinearity 35 (2022), no. 12, 5963-5992.

\bibitem{Par60}
W. Parry,
\textit{On the $\beta$-expansions of real numbers},
Acta Math. Acad. Sci. Hungar., 401-416, 1960.

\bibitem{Pen06} L. Peng,
\textit{Dimension of sequences of low complexity},
J. Math. (Wuhan) 26 (2006), no. 2, 133-136.

%\bibitem{PU10} F. Przytycki, M. Urba\'{n}ski,
%\textit{Conformal fractals: ergodic theory methods},
%London Mathematical Society Lecture Note Series, 371. Cambridge University Press, Cambridge, 2010.

\bibitem{Ren57} A. R\'{e}nyi,
\textit{Representations for real numbers and their ergodic properties}, Acta Math. Acad. Sci. Hungar., 477-493, 1957.

\bibitem{Schmeling1997} J. Schmeling,
\textit{Symbolic dynamics for $\beta$-shifts and self-normal numbers},
Ergodic Theory Dynam. Systems 17 (1997), no. 3, 675-694.

\bibitem{Shm19} P. Shmerkin,
\textit{On Furstenberg's intersection conjecture, self-similar measures, and the $L^q$ norms of convolutions}, Ann. of Math. (2) 159 no. 2 (2019), 319-391.

\bibitem{SW91} M. Shub, B. Weiss,
\textit{Can one always lower topological entropy}?
Ergodic Theory Dynam. Systems 11 (1991), 535-546.

\bibitem{Sun21} P. Sun,
\textit{Equilibrium states of intermediate entropies},
Dyn. Syst. 36 (2021), no. 1, 69-78.

\bibitem{Wa82} P. Walters,
\textit{An introduction to ergodic theory},
Grad. Texts in Math., vol. 79, Springer-Verlag, New York/Berlin, 1982.

\bibitem{Wu19}  M. Wu,
\textit{A proof of Furstenberg's conjecture on the intersections of $\times p$- and $\times q$-invariant sets},
Ann. of Math. (2) 189 (2019), no. 3, 707-751.

\end{thebibliography}
\end{document}